\documentclass{amsart}
\usepackage[all]{xy}
\usepackage{amsmath, amssymb, graphics, graphicx}
\DeclareGraphicsExtensions{.png}
\newcommand{\mathsym}[1]{{}}

\usepackage[mathscr]{eucal}
\usepackage[applemac]{inputenc}
\usepackage{graphicx}
\usepackage{hyperref}

\newcommand{\N}{\mathbb{N}}

\newcommand{\R}{\mathbb{R}}

\newcommand{\e}{\mathcal{E}}
\def\R{\mathbb{R}}
\def \N {\mathbb{N}}

\def \inc {\mbox{in\,}}

\def \Fix{{\rm Fix}}
\def\Co{\check{C}_{0}}
\def\F{{\check\pi_0}}

\newtheorem{lemma}{ \kern \parindent Lemma}[section]

\newtheorem{definition}{ \kern \parindent Definition}[section]
\newtheorem{proposition}{ \kern \parindent Proposition}[section]
\newtheorem{corollary}{ \kern \parindent Corollary}[section]
\newtheorem{theorem}{ \kern \parindent Theorem}[section]
\newtheorem{remark}{ \kern \parindent Remark}[section]

\def\Id{\mathrm{Id}}

\def\d{{d}}

\def\BG{{\rm BG}}

\def\e{\mathbf{c}}
\def\l{\mathbf{l}}
\def\r{\mathbf{r}}

\def\t{\mathbf{t}}
\def\d{\mathbf{d}}
\def\F{\check{\pi}_0}

\def\C{\mathbb{C}}

\begin{document}

\title[Regions of attraction, limits and end points\ldots]{Regions of attraction, limits and end points of an exterior discrete semi-flow}

%\author[Garc\'{\i}a-Calcines, Hern\'andez, Mara\~n\'on, Rivas]{J. M. Garc\'{\i}a-Calcines, L. J. Hern\'andez, M. Mara\~n\'on and M. T. Rivas}

\author[J.M. Garc\'{\i}a]{J.M. Garc\'{\i}a Calcines}
\address{Dto. de Matem\'{a}tica Fundamental \\
Universidad de La Laguna \\ 38271 LA LAGUNA.}
\email{jmgarcal@ull.es}

\author[L.J. Hern\'andez]{L.J. Hern\'andez Paricio}
\address{Dpto. de Matem\'{a}ticas y Computaci\'{o}n \\
Universidad de La Rioja \\
26004 LOGRO\~{N}O.} \email{luis\_javier.hernandez@unirioja.es}

\author[M. Mara\~n\'on]{M. Mara\~n\'on Grandes}
\address{Dpto. de Matem\'{a}ticas y Computaci\'{o}n \\
Universidad de La Rioja \\
26004 LOGRO\~{N}O.} \email{miguel.maranon@alum.unirioja.es}

\author[M.T. Rivas]{M.T. Rivas Rodr\'{\i}guez}
\address{Dpto. de Matem\'{a}ticas y Computaci\'{o}n \\ Universidad de La Rioja \\
26004 LOGRO\~{N}O.} \email{maria\_teresa.rivas@unirioja.es}



\thanks{Partially supported by the University of La Rioja
PROFAI13/15
%The third author has been partially supported by
and an FPI grant from the Government of La Rioja.}

\maketitle

\begin{abstract}

An exterior space is a topological space provided with a
quasi-filter of open subsets (closed by finite intersections). In
this work, we analyze some relations between the notion of an
exterior space and the notion of a discrete semi-flow.

On the one hand, for an exterior space, one can consider limits,
bar-limits and different sets of end points (Steenrod, \v Cech,
Brown-Grossman). On the other hand, for a discrete semi-flow, one
can analyze fixed points, periodic points, omega-limits, et
cetera.

In this paper, we introduce a notion of exterior discrete semi-flow, which is a mix of exterior space and discrete semi-flow.
 We see that a discrete semi-flow can be provided with the structure of an exterior discrete semi-flow by taking
as structure of  exterior space  the family of right-absorbing open subsets,  which  can be used to study
the relation between limits and periodic points and connections
between bar-limits and omega-limits. The different notions of end points
are used to decompose the region of attraction of an
exterior discrete semi-flow as a disjoint union of basins of end
points.  We also analyze the exterior discrete semi-flow structure induced by the  family of  open neighborhoods of a given sub-semi-flow.

\end{abstract}

\section{Introduction}

The homotopy theory of exterior spaces is related to the theory of
proper homotopy, shape and strong shape theory and the homotopy
theory of pro-spaces and simplicial sets. These theories  are
interrelated and some homotopy invariant groups are usually
considered in order to distinguish spaces or to characterize some
homotopy equivalences. For instance, Steenrod, Borsuk-\v Cech and
Brown-Grossman homotopy groups can be given in all these contexts.
The theory of exterior spaces
 has proven to be very useful in the study of topological aspects of
several settings such as proper homotopy theory and its numerical
invariants or shape theory (\cite{DHR09,EHR89,EHR05,GGH98,GGH04}).

Recently, some applications of the category of exterior spaces to
the study of continuous dynamical systems (flows) have been
developed, see \cite{GHR12}. The notion of an absorbing open
subset of a dynamical system (i.e., an open subset that contains
the ``future part'' of all the trajectories) gives one of the keys
to connect the theory of dynamical systems and the theory of
exterior spaces. The family of all absorbing open subsets is a
quasi-filter which gives  the structure of an exterior space to
the flow. The  limit space and end space of  an exterior space are
used to construct the limit spaces and end spaces of a dynamical
system. Using end points (of Freudenthal type, see
\cite{Freudenthal}) of  a dynamical system, one has an induced
decomposition of a dynamical system as a disjoint union of basins
of end points.

In \cite{GHR14}, the theory of exterior spaces has also been used
to  construct a $\Co^{\r}$-completion and a $\Co^{\l}$-completion
of a dynamical system. For a given flow $X$, two
% If $X$ is a flow, we construct canonical
 maps $X\to \Co^{\r}(X)$ and $X\to \Co^{\l}(X)$ are constructed
and, when one of  these maps is a
homeomorphism, one has  the class of $\Co^{\r}$-complete and
$\Co^{\l}$-complete flows, respectively.

There are many relations between the topological properties of the
completions and the dynamical properties of a given flow. In the
case of a complete flow, this gives  interesting relations between
its topological properties (separability properties, compactness,
convergence of nets, etc.) and its dynamical properties (periodic
points, omega-limits, attractors, repulsors, etc.). These results
confirm  the importance of the purely topological behavior of a
continuous dynamical system in many, radically different in
principle, situations (differential equations, non-linear
analysis, transformation groups, et cetera).  Some topological
techniques in dynamical systems were initially introduced by
H. Poincar\'e \cite{Poincare2, Poincare} and  G. D. Birkhoff
\cite{Birkhoff}.

Taking into account the deep relations, given in \cite{GHR12,
GHR14}, between exterior spaces and continuous dynamical systems,
the main general objective of our present  study is to find out
the relationships between exterior spaces and the dynamical
properties of discrete dynamical systems. Discretization processes
and  the suspension  of the Poincar\'e first return map and other
constructions give a nice interdependence of the properties of
discrete and continuous dynamical systems. As a consequence, many
of the properties, results and applications given in \cite{GHR12,
GHR14} must have a counterpart of notions and results that  can be
developed and proved  for discrete dynamical systems.

It is worth mentioning some differences between continuous flows
and discrete semi-flows. In a semi-flow, we are dealing with a
semi-group of continuous maps instead of a group of
homeomorphisms. This implies that the construction and properties
of left omega-limits is quite different to those of right
omega-limits.  A continuous flow has the nice property that all
the points in a trajectory are in the same path component;
however, we can not ensure this nice property for discrete
semi-flows. These differences have to be taken into account when
one analyzes the interrelations between the theory of exterior
spaces and the theory of discrete semi-flows.  Subsequently, one
has some similitudes with the results and tools given in
\cite{GHR12}, but   new (non-analogous) techniques ought to be
developed for a better  analysis of discrete semi-flows. For
instance, it is interesting to remark that, for continuous
dynamical systems, an analogue of the Borsuk-\v Cech invariant
$\check \pi_0$ has played an important role in order to divide a
continuous flow into a disjoint union of  basins of end points
given by the functor   $\check \pi_0$. Nevertheless, for discrete
semi-flows it is better to change the analogue of functor $\check
\pi_0$ for the analogue of the Brown-Grossman invariant
$\pi_0^{BG}$. The use of Brown-Grossman end points is more natural
for discrete semi-flows and this technique was not introduced in
the previous applications of exterior spaces to continuous flows.

On the one hand, in this paper, for exterior discrete semi-flows, we have introduced the region of
attraction of an externology, the limit and bar-limit of an
externology and different notions of end points;  at least, one
can consider  three types of end points which are given by the
analogues of the $0$-dimensional homotopy invariants of types:
Borsuk-\v Cech $\check \pi_0$, Steenrod $\pi_0^S$ and Brown-Grossman groups $\pi_0^{BG}$.

On the other hand, in dynamic discrete systems one has the usual
notion of region of attraction of a right-invariant subset, the
omega-limit of a point, periodic points, basins of $n$-cycles, et
cetera.

The techniques presented in this paper give a very nice connection
between notions associated to an exterior space and dynamic
notions associated to a discrete semi-flow. The regions of
attraction of an externology are related to regions of attraction
of a right-invariant subset (Theorems \ref{regionattraction},
\ref{regionattraction1}), the notion of limit is related to the
subset of periodic points (Theorem \ref{periodic}), the bar-limit
is connected to the notion of omega-limit (Theorem \ref{main}),
the basin of an end point of Borsuk- \v Cech type is related with
the basin of a fixed point and the basin of an  end point of
Brown-Grossman type is related with the basin of a periodic point
and the basin of an $n$-cycle.

There are other interesting interrelations of exterior spaces and
discrete semi-flows that have not been analyzed in our first
approach to the study of discrete semi-flows via exterior spaces.
But the use of adequate externologies will also permit us to study
questions related to sensibility at the initial conditions,
stability problems and other dynamical questions connected to
higher dimension homotopy groups of exterior spaces.

\section{Preliminaries}

\subsection{ Exterior spaces}\label{exteriorspaces}

In this subsection we recall some definitions and properties
related to exterior spaces. For a deeper study of the category of
exterior spaces we refer the reader to \cite{G98}, \cite{GHR09}.

Given a topological space $X$ with a topology $\t_{X}$ and a
subset $A\subset X$, the closure of $A$ in $X$ is denoted by
$\overline{A}$ and the interior by $\mathring A$ or ${\rm
Int}(A)$.

\begin{definition}

Let $X$ be a topological space. An externology on $X$ is a
non-empty collection $\varepsilon(X) $ of open subsets which is
closed under finite intersections and such that, if $E\in
\varepsilon(X)$ and  $U$ is an open subset  such that $E\subset
U$, then $U\in \varepsilon(X)$. If an open subset is a member of
$\varepsilon(X) $, then it  is said to be an exterior open subset.

An {\bf exterior space}  $(X,\varepsilon(X))$ consists of a
topological  space $X$ together with an externology
$\varepsilon(X) $. When it is clear from the context,  we will
shorten the notation by just writing   $X$ instead of
$(X,\varepsilon(X))$.

A map $f:(X,\varepsilon(X))\rightarrow (X',\varepsilon'(X'))$ is
said to be an {\bf  exterior map} if it is continuous and
$f^{-1}(E')\in \varepsilon(X) $, for all $E'\in \varepsilon '(X).$

\end{definition}

For a topological space $X$ we can consider the co-compact
externology $\varepsilon^{\e}(X)=\{E\subset X \mid X\setminus E
\mbox{ is closed compact}\}$. We denote $\R$, $\R_+$ and $\N$ the
exterior spaces determined by the usual topology and co-compact
externology in the sets of real numbers $\R$, non-negative real
numbers $\R_+$ and natural numbers $\N=\{0,1,2,\dots\}$,
respectively.

The category of exterior spaces and exterior maps is denoted by {\bf E},
and the category of topological spaces and continuous maps by {\bf Top}.

We can consider the functor
$$(\cdot)\bar \times (\cdot) \colon {\bf E}\times {\bf Top}\rightarrow {\bf E}$$
given by the following construction.
Let $(X, \varepsilon(X))$ be an exterior space and $Y$ a topological space.
 We take on
$X\times Y$ the product topology  and the externology
$\varepsilon({X\bar \times Y})$ given by those open subsets
$E\subset X\times Y$ such that, for each $y\in Y$, there exists an open neighborhood
 $U_y$ and $T^{y}\in \varepsilon (X)$ such that
$T^{y}\times U_{y}\subset E.$ In order to avoid a possible
confusion with the product externology, this exterior space will
be denoted by $X\bar{\times }Y.$

Given $f,g\colon X \to Y$ in {\bf E},  it is said that $f$ is
exterior homotopic to $g$ if there is an exterior homotopy  $H
\colon X\bar \times I \to Y$ from  $f$ to $g$. Denote by $\pi {\bf
E}$ and $ \pi{\bf Top}$  the exterior homotopy category and the
usual homotopy category corresponding  to  ${\bf E}$ and $ {\bf
Top}$, respectively.  Given $X, Y$ two exterior spaces, the set of
exterior homotopy classes from $X$ to $Y$ will be denoted by $\pi
{\bf E}(X,Y)$. Its elements  $[f]$ are homotopy equivalence
classes of exterior maps $f\colon X \to Y$. Similar notation is
used in the case  of topological spaces.

\medskip

We recall that, for a topological space $Y$, $\pi_0(Y)$ denotes
the set of path components of $Y$ and we have a quotient map $Y\to
\pi_0(Y)$. Note that a continuous map $f\colon Y \to Y'$ induces a
natural map $\pi_0(f) \colon \pi_0(Y) \to \pi_0(Y')$.

An inverse system of sets (or topological spaces) is a functor
$Z\colon I \to {\bf Sets}$, where $I$ is a directed set and ${\bf
Sets}$ is the category of sets (resp., considering  {\bf Top}).
The functor $Z$ carries $i\geq j$, $i,j \in I$, to $Z_j^i \colon
Z_i \to Z_j$. The inverse limit of $Z=\{Z_i\}$ is denoted by $\lim
_{i\in I}Z_i$ (or just by $\lim Z_i$). An element of the inverse
limit can be represented by an element $(z_i)_{i \in I}$ of the
product $\prod_{i \in I}Z_i$ satisfying that $Z_j^i(z_i)=z_j$,
$i\geq j$.

 We also have that if $\{Y_i\}$ is an inverse
system of topological spaces, then $\{\pi_0(Y_i)\}$ is an inverse
system of sets and one can consider the inverse limit $\lim_{i\in
I} \pi_0(Y_i)$. An element of $\lim_{i\in I} \pi_0(Y_i)$ is given
by $(C_i)_{i \in I}$, where $C_i$ is a path component of $Y_i$
such that  $Y_j^i(C_i ) \subset C_j$ for $i \geq j$. For more
results and properties about inverse systems, we refer the reader
to \cite{E-H}.

\begin{definition} Given an exterior space $X$, its externology
$\varepsilon(X)$ can be considered as an inverse system of
topological spaces, and we have the following notions:

The topological subspace
$$L(X)=\lim_{E\in\varepsilon(X)}E=\bigcap_{E\in\varepsilon(X)}E$$
will be called the {\bf limit space} of $X$ and
$$\bar L(X)=\lim_{E\in\varepsilon(X)}\overline E=\bigcap_{E\in\varepsilon(X)}\overline E$$
will be called the {\bf bar-limit space} of $X$.

The {\bf $\check \pi_0$-end set} of $X$ is given by
$$\check \pi_0(X)=\lim_{E\in\varepsilon(X)}\pi_{0}(E).$$

The {\bf  ${\check {\bar \pi}}_0$-end set} of $X$ is given by
$$ {\check {\bar \pi}}_0(X)=\lim_{E\in\varepsilon(X)}\pi_0(\overline E).$$

The {\bf $\pi_0^{\BG}$-end set} of $X$ is given by
$$\pi_0^{\BG}(X)=\pi {\bf E}(\N,X).$$

The {\bf $\pi_0^{\rm S}$-end set} of $X$ is given by
$$\pi_0^{\rm S}(X)=\pi {\bf E}(\R_+,X).$$

\end{definition}

\subsection{Discrete semi-flows}

Next, we recall some basic notions about discrete semi-flows.
These notions can be given for a set or for a topological space.

\begin{definition}
A {\bf discrete semi-flow} on a (topological  space) set  $X$ is a
(continuous) map $\phi \colon \N {\times} X\to  X$ such that:

\begin{enumerate}
\item[(i)] $\phi(0,x)=x$, $\forall x \in X$;
\item[(ii)] $\phi(n,\phi(m, x))=\phi( n+m, x)$, $\forall x \in X$, $\forall n,m \in \N$.
\end{enumerate}
Note that giving a discrete semi-flow $(X,\phi)$ on a (topological
space) set $X$ is equivalent to giving a (continuous) map
$f=\phi^1\colon X \to X$.

A discrete semi-flow on $X$ will be denoted by $(X, \phi)$ and,
when no confusion is possible, we will use $X$ and  $n\cdot
x=\phi(n,x)$  for short. Similarly, for a subset $S\subset \N$ we
will denote $S\cdot x =\{n \cdot x \mid n \in S\}$.

Given two discrete semi-flows  $(X, \phi)$ and $(Y,\psi)$, a {\bf
discrete semi-flow morphism} $f\colon (X, \phi) \to (Y, \psi)$ is
a (continuous) map $f\colon X \to Y$ such that $f(n\cdot x)=n\cdot
f(x)$, for every $(n,x)\in \N\times X$. The category of discrete
semi-flows (defined on topological spaces) will be denoted by ${\bf F(\N})$.
\end{definition}

Given a discrete semi-flow  $\phi \colon \N {\times} X\to  X$,
$n_0\in \N$, $x_0\in X$, we have the induced maps
$\phi^{n_0}\colon X\to X$, $\phi^{n_0}(x)=\phi(n_0,x)$ and
$\phi_{x_0}\colon \N\to X$, $\phi_{x_0}(n)=\phi(n,x_0)$.

For  a discrete semi-flow  $(X, \phi)$, a subset $A\subset X$ is
said to be {\bf right-invariant} if $\phi^1(A)\subset A$ and it is
said to be {\bf left-invariant} if $(\phi^1)^{-1}(A)\subset A$. A
subset which is left-invariant and right-invariant it is said {\bf
completely invariant}.

Given two points $x,y \in X$, we have the following equivalence
relation: $x\sim y$ if there are $k,l \in \N$ such that
$\phi^k(x)=\phi^l(y)$. If $[x]$ is the equivalence class of $x$,
note that $[x]$ is a completely invariant subset. Denote by
$X/\hspace{-3pt}\sim$ the quotient set which has a trivial induced
action. The subset $[x]$ of $X$ is the {\bf big orbit} of $x$ and
the subset $\N \cdot x$ is the {\bf trajectory} of $x$.

\begin{definition} Let $X$ be a discrete semi-flow and $x$ a point of $X$.
\begin{itemize}
\item[(i)] $x$ is  a {\bf fixed} point  if, for every $n\in \N$,  $n\cdot x =x$.
\item[(ii)] $x$ is a {\bf periodic or cyclic} point if there is $n\in \N$, $n\not = 0$, such that $n\cdot x= x$.
\item[(iii)] $x$ is a {\bf $m$-periodic  point} if $m\cdot x= x$.
\item[(iv)] $x$ is a {\bf $m$-cyclic point} if $m\cdot x= x$ and if  for $0<k<m$, then $k\cdot x\not = x$.
\end{itemize}
 \end{definition}

The right-invariant  subsets of fixed,  periodic, $m$-periodic and
$m$-cyclic points of $X$ are denoted by $\Fix(X)$, $P(X)$,
$P_m(X)$ and $C_m(X)$, respectively. From the definition, it is
clear that $C_m(X)\subset P_m(X)$.

\medskip

A net  of a topological  space $X$ is denoted by $x_i$, where  we
suppose that $i$ describes a directed set. In this paper, $[n,
+\infty)$ denotes the subset $\{m\in \N\mid m\geq n\}$ and
similarly $(n, +\infty)=\{m\in \N\mid m>n\}$.
 The following notions are given for topological spaces with a given semi-flow structure:

\begin{definition}\label{rightlimit} For a discrete semi-flow  $(X,\phi)$,
the {\bf omega-limit set of a point} $x \in X$ (or right-limit
set, or positive limit set) is given as follows:
$$\Lambda (x) =\{y \in X \mid \exists  n_i \to +\infty, \, n_i \in \N, \, \mbox{such that }
 n_i\cdot x \to y\}.$$

We note that the subset $\Lambda(x)$ admits the alternative
definition
$$\Lambda (x) = \bigcap_{n \geq 0}\overline{ [n, +\infty) \cdot x}.$$
For a given subspace $S \subset  X$, the set $\Lambda (S)
=\bigcup_{x\in S} \Lambda(x)$ is called the  {\bf omega-limit set}
of $S.$

A point $x \in X$ is said to be {\bf Poisson stable} (or
positively Poisson stable) if $x\in \Lambda (x).$ We will denote
by ${\rm Poisson}(X)$ the right-invariant subset of Poisson stable
points of $X$.
\end{definition}

Note that $P(X) \subset {\rm Poisson}(X)  \subset \Lambda (X)$. It is also easy to check that  $\Lambda(x)$ and  $\Lambda(X)$ are
right-invariant  subsets of $X$ and one also has that, for $x\in
X$, $x$ and $\phi^1(x)$ have the same omega-limit.

\begin{lemma} Let $(X, \phi)$ be a discrete semi-flow and let $S\subset X$. Then,
\begin{itemize}
\item[(i)]$ \phi^1(\Lambda (x))\subset \Lambda (x)$,
\item[(ii)]$ \phi^1(\Lambda (S))\subset \Lambda (S)$,
\item[(iii)]$\Lambda (\phi^1(x)) = \Lambda (x)$.
\end{itemize}
\end{lemma}

\begin{remark}
Observe that, if $X$ satisfies the first axiom of countability
(for instance, when $X$ is metrizable), then we can consider
sequences instead of nets in definition \ref{rightlimit}.
\end{remark}

\begin{definition}  \label{LagrangeStable} Let $(X,\phi)$ be a discrete semi-flow and $x\in X$.
It is said that $X$ is {\bf Lagrange stable at $x$} (or positively
Lagrange stable at $x$) if $\overline{\N \cdot x}$ is a compact
subset. $X$ is {\bf  Lagrange stable} if, for every $x\in X$, $X$
is  Lagrange stable at $x$.
\end{definition}

\begin{definition}\label{region}Let $(X, \phi)$ be a discrete semi-flow and let
$S\subset X$.

The \textbf{region of pseudo-attraction of $S$} is defined by
$$\text{PA}(S)=\{x \in X \mid \Lambda(x)\subset S\}.$$

The \textbf{region of weak-attraction of $S$} is defined by
$$\text{WA}(S)=\{x \in X \mid \Lambda(x)\cap S \not= \emptyset\}.$$

The \textbf{region of attraction of $S$} is defined by
$$\text{A}(S)=\{x \in X \mid \emptyset \not =\Lambda(x)\subset S\}.$$
\end{definition}

%Obviously, $\text{A}(S) \subset \text{PA}(S)$ and $\text{A}(S)
%\subset \text{WA}(S).$

\begin{lemma}\label{attraction}Let $(X, \phi)$ be a discrete semi-flow and let $S\subset X$.  Then,
\begin{itemize}
\item[(i)]$\text{PA}(S)=\text{PA}(\emptyset) \sqcup \text{A}(S)$;
\item[(ii)]$\text{A}(S)=\text{PA}(S) \cap \text{WA}(S)$;
\item[(iii)]$\text{PA}(S),\text{WA}(S),\text{A}(S)$ are completely invariant;
\item[(iv)] If $X$ is Lagrange stable at every point of $\text{PA}(S)$, then
$$\text{PA}(\emptyset)=\emptyset \hspace{10pt}\mbox{and}\hspace{10pt}\text{PA}(S)=
\text{A}(S).$$
\end{itemize}
\end{lemma}

\begin{proof} (i) and (ii) follow from the definition and (iii) from the identity
$\Lambda(x)= \Lambda(\phi^1(x))$, for every $x \in X$. (iv)
follows from the Lagrange stability of $X$ and  the fact  that
$\Lambda(x)$ is a non-empty compact subset of $X$,  for any $x\in
X$.
\end{proof}

\section{Natural transformations for limit and end spaces of exterior spaces}\label{limitsandnt}

In subsection \ref{exteriorspaces}, for an exterior space, the
notions of limit, bar-limit and different sets of end points were
introduced. Now, we analyze some relationships between them.

Given an exterior space $(X, \varepsilon(X))$, the canonical maps
$$E\subset \overline{E},\quad E\to \pi_{0}(E), \quad \overline{E}\to \pi_{0}(\overline{E})$$
induce a commutative diagram

$$ \xymatrix{L(X) \ar[r]^{e} \ar[d] & {\check \pi}_{0}(X) \ar[d]\\
{\bar L}(X)\ar[r]^{\bar e}& {\check {\bar \pi}}_{0}(X)}$$

We can consider the shift map $s\colon \N \to \N$, given by
$s(i)=i+1$, $i\in \N$. This exterior map induces the canonical map
$S \colon \pi_{0}^{\BG}(X) \to  \pi_{0}^{\BG}(X)$,
$S([\alpha])=[\alpha s]$, for a given exterior map $\alpha \colon
\N \to X$.

The inclusion map $\inc\colon\N\to\R_+$ induces a natural
transformation $$R_X \colon \pi_0^{\rm S}(X) \to \pi_0^{\BG}(X),$$
given by $R_X([\alpha])=[\alpha|_{\N}]$ for an exterior map
$\alpha \colon \R_+ \to X$. Then, we can consider the diagram
$$\xymatrix{\pi_0^{\rm S}(X)\ar[rr]^{R_X}  & & \pi_0^{\BG}(X) \ar@<1ex>[r]^{\Id} \ar@<-1ex>[r]_{S} & \pi_0^{\BG}(X),
}$$ where the image of $R_X$ is the equalizer of the identity and
shift maps (see \cite{GHR09}).

\begin{remark} In fact, there is a very interesting exact sequence
of homotopy groups associated to an exterior space $X$ with a base
ray $\alpha \colon \R_+ \to X$
$$\xymatrix{\cdots \ar[r] &\pi_{q}^{\rm S}(X)\ar[r]   &\pi_{q}^{\BG}(X) \ar[r]^{\Id-S}  & \pi_{q}^{\BG}(X) \ar[r]  & \pi_{q-1}^{\rm S}(X) \ar[r]
& \cdots}$$ ending at dimension zero in the above diagram. These
higher exterior homotopy invariants are  powerful tools for the
study and classification of exterior spaces, see \cite{E-H},
\cite{EHR05}, \cite{GHR09}. In the next sections, we will only
consider the zero dimensional part of this sequence for the study
of end points (and their basins) of an exterior discrete
semi-flow.
\end{remark}

The connection between the set of $\pi_0^{\rm S}$-end points and
the set of $\check \pi_0$-end points of an exterior space is
given by the natural transformation
$$\eta_X\colon \pi_0^{\rm S}(X) \to \F(X),$$
where $\eta_X$ is defined by $\eta_X([\alpha])=\check
\pi_0(\alpha)(+\infty)$ (observe that $\check \pi_{0}(\R_{+} )
=\{+\infty\}$).

\medskip

There is also a relationship between  the sets of $\check
\pi_0$-end points and $\pi_0^{\BG}$-end points of  $X$, but we
need exterior spaces with some additional conditions.

\begin{definition}   An exterior space $(X,\varepsilon(X))$ is said to be first-countable at infinity
if $\varepsilon(X)$ contains a countable base $E_{0}\supset
E_{1}\supset E_{2}\supset \cdots $; that is, each  $E_i \in
\varepsilon(X)$ and,   for every $E\in\varepsilon(X)$, there is
$i\in \N$ such that $E_{i}\subset E$.
\end{definition}

\begin{proposition}\label{pipi} Let $(X,\varepsilon(X))$ be a first-countable at infinity exterior space. Then,
there is a canonical injective map $\theta_X \colon\check \pi_{0}(X) \to  \pi_{0}^{\BG}(X)$  such that the image of this map
is  the set $\{a \in  \pi_{0}^{\BG}(X)\mid S(a)=a\}$. Moreover, $\eta_X$ is surjective and the map $R_X$ factorizes as
$$\xymatrix{\pi_0^{\rm S}(X)\ar[rr]^{R_X} \ar@{>>}[rd]_{\eta_X} & & \pi_0^{\BG}(X)  \\
&\F(X) \ar@{>->}[ru]_{\theta_X}}$$ Consequently, for exterior
spaces which are first-countable at infinity, $\theta_X\colon
\F(X) \to \pi_0^{\BG}(X) $ is the equalizer of the identity and
the shift map.
\end{proposition}

\begin{proof} Since $X$ is  first-countable at infinity, we have a countable base:
$E_{0}\supset E_{1}\supset E_{2}\supset \cdots $. By the definiton
of $\check \pi_0(X)$, it follows that $\check \pi_0(X) = \lim
\pi_0(E_i)$. Therefore an end point $a \in \check \pi_0(X)$ can be
represented  by a sequence of path components $C_{0}\supset
C_{1}\supset C_{2}\supset \cdots$. For every $i \in \N$, take a
point $x_{i}\in C_{i}$ and define a map $\alpha\colon \N \to X$ by
$\alpha(i)=x_{i}$. Then, the map $\theta_X \colon \check
\pi_{0}(X) \to  \pi_{0}^{\BG}(X)$ is given by $\theta_X
(a)=[\alpha]$. It is easy to check that the definition of
$\theta_X$ does not depend on the  chosen base nor the  chosen
points and that it is injective.

Now  if  $x_i,x_{i+1}\in C_i$, then there is path $F_i\colon I \to
C_i$ from $x_i$ to $x_{i+1}$. The exterior homotopy $F \colon \N
\bar \times I \to X$ given  by $F(i, t)=F_{i}(t)$ satisfies that
$F(i,0)=x_{i}$ and $F(i,1)=x_{i+1}$. This implies that $S(
\theta_X (a))= \theta_X (a)$. Conversely, take $\alpha \colon \N
\to X$ and suppose that $S([\alpha])=[\alpha s]=[\alpha]$. Then,
there is an exterior homotopy $F\colon \N \bar \times I \to X$
from $\alpha $ to $\alpha s$ and we can define an exterior map
$\beta \colon  \R_{+} \to X$ by $\beta(r)= F(E(r), r-E(r))$, where
$E(r)$ is the integer part of $r\geq 0$.  It is easy to check that
$\theta_X (\check  \pi_{0} (\beta) (+\infty )) =[\alpha]$. The
rest of the proof follows straightforward.
\end{proof}

\begin{remark} \label{bar-exterior} In order to obtain a more complete description of the relationships between constructions and invariants associated to an exterior space and distinguished subsets of discrete semi-flows the authors
think that some further work  needs  be done using the following new notions: Given two exterior spaces $(X, \varepsilon(X)), (X', \varepsilon(X'))$ a bar-exterior map $f:
(X, \varepsilon(X)) \to (X', \varepsilon(X'))$ is a continuous map
$f\colon X \to X'$ such that for every $E'\in \varepsilon(X')$,
there is $E\in \varepsilon(X)$ such that $f(\overline{E})\subset \overline{E'}$. Note that an exterior map is always a bar-exterior map. This yields to a large category
${\bf \bar E}$, a new homotopy category $\pi({\bf \bar E})$
and new invariants:

The {\bf ${\bar\pi}_0^{\BG}$-end set} of $X$ is given by
${\bar\pi}_0^{\BG}(X)=\pi {\bf (\bar E)}(\N,X).$

The {\bf ${\bar \pi}_0^{\rm S}$-end set} of $X$ is given by
${\bar\pi}_0^{\rm S}(X)=\pi ({\bf\bar E})(\R_+,X).$

In the present paper, we add some remarks about some possible new results connected with the notion of bar-exterior map. The new notions are related to pure topological descriptions of Julia and Fatou sets and  a more complete study needs to be done to clarify these interesting connections.

\end{remark}

\begin{remark} It is interesting to observe that taking the
categories ${ \bf\bar E}$ and $\pi ({ \bf\bar E})$,
there is a new exact sequence
of homotopy groups associated to an exterior space $X$ with a base
ray $\alpha \colon \R_+ \to X$
$$\xymatrix{\cdots \ar[r] &{\bar\pi}_{q}^{\rm S}(X)\ar[r]   &{\bar\pi}_{q}^{\BG}(X) \ar[r]^{\Id-S}  &{\bar \pi}_{q}^{\BG}(X) \ar[r]  &{\bar \pi}_{q-1}^{\rm S}(X) \ar[r]
& \cdots}$$
The new ``bar-invarinants" can be related  to invariants above
using the exact sequences associated
to the pairs $(\overline{E}, E)_{E\in \varepsilon(X)}$ that will
reflect the differences between both type of invariants.
\end{remark}

\section{Exterior discrete semi-flows}

In this section, we introduce the notion of exterior discrete semi-flow, which combines the
notions of exterior space and discrete semi-flow.

\begin{definition}  Let $X$ be an exterior space.
An  {\bf exterior discrete semi-flow } is a discrete semi-flow
$\phi \colon \N {\times}  X \to  X$ such that, for any $n\in \N$,
$\phi^{n}\colon X \to X$ is exterior (this is equivalent to the
simpler condition $\phi^{1}\colon X \to X$ is  exterior).

An exterior discrete semi-flow $X$ is said to be  {\bf
$\d$-exterior} if it satisfies that $\phi_x \colon \N \to X$ is
exterior, for every $x\in X$.

An {\bf exterior discrete semi-flow  morphism} of exterior
discrete semi-flows $f \colon X \to Y$ is a discrete semi-flow
morphism  such that $f$ is exterior.
\end{definition}

\begin{remark} The condition which asserts that  $\phi_x \colon \N \to X$
is exterior for every $x\in X$ can be replaced
by the requirement that  $\phi \colon \N\bar\times X_{\d} \to  X$
is exterior, where $X_{\d}$ denotes the set  $X$ provided with the
discrete topology. This fact justifies the use of the term
``$\d$-exterior".
\end{remark}

Denote by $\bf{EF(\N)}$ the category of exterior discrete
semi-flows and by  $\bf{E^{\d}F(\N)}$, the full subcategory  of
$\d$-exterior discrete semi-flows.

We adopt the following notational convention: an exterior discrete
semi-flow will be denoted by a triplet $(X, \phi,
\varepsilon(X))$. Nevertheless, when the action $\phi$ or the
externology are clear on a determined context, we will shorten
the notation and we will use $(X,\varepsilon(X))$ or $(X,\phi)$;
moreover, in many cases the notation will be reduced to $X$.
\medskip

We have defined above the limit space of an exterior space. In
particular, since an exterior discrete semi-flow $X$ is an
exterior space, we can consider the limit space $L(X)$ and  the
bar-limit space $\bar{L}(X)$.

\begin{proposition} Let $(X, \phi, \varepsilon(X))$ be an exterior
discrete semi-flow. Then, $L(X), { \bar L}(X)$ are
right-invariant.
\end{proposition}

\begin{proof} Denote $f=\phi^{1}$. Since $f$ is an exterior map, it follows that
$f^{-1}(E) \in \varepsilon(X)$ for every $E\in \varepsilon(X)$.
Then, one has that $$f^{-1}(L(X))=f^{-1}\left(\bigcap_{E\in
\varepsilon(X)} E\right)=\bigcap_{E\in  \varepsilon(X)}f^{-1}
(E)\supset \bigcap_{E\in  \varepsilon(X)} E=L(X).$$ This implies
that $f(L(X))\subset L(X)$. For $ {\bar L}(X)$, the proof is
similar using also the fact that if  $f$ is continuous,
$f(\overline{E}) \subset \overline{f(E)}$ for every $E\in
\varepsilon(X)$.
\end{proof}

For an exterior discrete semi-flow $X$, one has that the
exterior map $\phi^{1}\colon X \to X$ induce the maps
$\pi_{0}^{S}(\phi^{1})$,
$\pi_{0}^{\BG}(\phi^{1})$,
%$\bar{\Pi}_{0}(f)$,
${\check \pi}_{0}(\phi^{1})$, ${\check {\bar \pi}}_{0}(\phi^{1})$ that give canonical
discrete semi-flow structures on the corresponding sets $\pi_{0}^{S}(X)$,  $\pi_{0}^{\BG}(X)$,
%$\bar{\Pi}_{0}(X)$,
${\check \pi}_{0}(X)$, ${\check {\bar \pi}}_{0}(X)$. These
exterior homotopy invariants can be taken as a set together with a
discrete semi-flow structure.

\section{The regions of attraction of an exterior discrete semi-flow}

Now, suppose that $(X,\phi, \varepsilon(X))$ is an exterior
discrete semi-flow. Consider $$D(X)=\{x\in X| \phi_{x} \mbox{ is
exterior}\}.$$ In a similar way, one can define $\bar D(X)$ as the
set of points $x\in X$ such that, for every $E\in \varepsilon(X)$,
there is $n_E\in \N$ such that, for every $n\geq n_E$, $\phi^n(x)
\in \overline{E}$. The subspace $\bar D(X)$ has many similar
properties to those of the subspace $D(X)$, whose proofs are left
to the reader. However, the most    sophisticated proofs  will be
included.

\begin{proposition} For  an exterior discrete semi-flow
$X=(X,\phi, \varepsilon(X))$, the subspace $D(X)$ is completely
invariant and, with the relative externology (i.e., the family of
intersections of the form $D(X)\cap E$, $E \in \varepsilon(X)$),
is a $\d$-exterior discrete semi-flow.
\end{proposition}

\begin{proof} Denote $D=D(X)$. Observe that $\phi_{\phi^{1}(x)}= \phi_{x} \,s$,
where the shift $s$, $s(n)=n+1$, is an exterior map. This implies
that, if $x\in D$, then  $\phi^1 (x) \in D$.

Now, in order to prove that $D$ is left-invariant,  suppose that
$\phi^{1}(x)=y$ and $y\in D$. Given an exterior open subset $E$,
there is $n\geq 1$ such that $\phi^{m}(y)\in E$, for every $m\geq
n$. Then, $\phi^{k}(x)\in E$, for every $k\geq n+1$. This implies
that $x\in D$.

Finally, since $X$ is an exterior discrete semi-flow, taking on
$D$ the relative externology, it follows  that  $\phi^{n}\colon D
\to D$ is exterior, for every $n\in \N$. By the definition of $D$
one has that $\phi_x \colon D \to  D$ is exterior, for every $x
\in D$. Therefore, $D$ with the relative externology is a
$\d$-exterior discrete semi-flow.
\end{proof}

\begin{definition} Suppose that $X=(X,\phi, \varepsilon(X))$ is an exterior discrete semi-flow.
The $\d$-exterior discrete semi-flow $D(X)$ is said to be {\bf the
region of pseudo-attraction} of $\varepsilon(X)$. The exterior
discrete semi-flow $\bar D(X)$ is said to be {\bf the region of
pseudo-bar-attraction} of $\varepsilon(X)$.
\end{definition}

\begin{remark} Notice that  the canonical functor $\bf{EF(\N)} \to\bf{E^{\d}F(\N)}$, which
carries  $X$ to $D(X)$, is right adjoint to the inclusion functor
$\bf{E^{\d}F(\N)}  \to \bf{EF(\N)}.$
\end{remark}

\begin{remark} Using the notion of bar-exterior map mentioned in Remark \ref{bar-exterior} one has the following alternative definition of $\bar D(X)$ for an exterior discrete semi-flow $X$,
$\bar D(X)=\{x \in X \mid \phi_x\text{  is bar-exterior}\}$.
\end{remark}

If $X=(X,\phi, \varepsilon(X))$ is an exterior discrete semi-flow, the inclusion $D(X)\subset X$ of exterior spaces
induces the transformation $L(D(X)) \to L(X)$.

\begin{proposition}\label{eledeigualele} Let $X=(X,\phi, \varepsilon(X))$ be an exterior discrete semi-flow.
Then, $L(D(X))= L(X)$ (similarly, $\bar L(\bar D(X))= \bar L(X)$).
\end{proposition}

For an exterior discrete  semi-flow, some relations between  regions of attraction associated to its externology and
standard regions of attraction as a discrete semi-flow as the following:

\begin{theorem}\label{regionattraction}
Suppose that $X=(X,\phi, \varepsilon(X))$ is an exterior discrete semi-flow. Then,
\begin{itemize}
\item[(i)] $D(X) \subset  \bar D(X) \subset  \text{PA}(\bar L (X))$;
\item[(ii)] If $(X,\phi)$ is  Lagrange stable at every point in $ \bar D(X)$, then $\bar D(X) \subset \text{A}(\bar L (X))$;
\item[(iii)] If $(X,\phi)$ is Lagrange stable at every point in $\text{A}(L (X))$, then $\text{A}(L (X)) \subset D(X)$;
\item[(iv)] If $(X,\phi)$ is  Lagrange stable at every point in $\text{PA}(L (X))$, then $\text{PA}(L(X)) \subset D(X)$;
\item[(v)] If $(X,\phi)$ is  Lagrange stable at every point in $\text{PA}(L (X))$ and $L(X)=\bar L(X)$, then
$D(X)=\text{PA}(L (X))=A(L(X)) $.
\end{itemize}
\end{theorem}

\begin{proof}

\begin{enumerate}

\item[(i)] Obviously, $D(X) \subset \bar D(X)$. If $x \in \bar D(X)$ and $E\in \varepsilon(X)$,
there is $n_E \in \N$, such that $[n_E, \infty) \cdot x \subset
\overline{E}$. Then, $\overline{[n, \infty) \cdot x } \subset
\overline{E}$. This implies that $\Lambda(x) \subset
\overline{E}$, for every $E\in \varepsilon(X)$. Hence, $\Lambda(x)
\subset \bar L (X).$

\item[(ii)] If $(X,\phi)$ is Lagrange stable at a point $x$, we also have that $\Lambda(x)\not = \emptyset$.

\item[(iii)] Suppose that $x\in A(L (X))$ and $\phi_x$ is not exterior. Then, there is $E\in \varepsilon(X)$
and an increasing sequence $n_1<n_2< \cdots$ such that
$\phi^{n_k}(x) \in X\setminus E.$ Since $X$ is Lagrange stable at
$x$, there is a subnet $n_{k_i}$ such that $n_{k_i} \cdot x \to y
\in X \setminus E$. But we also have that $y \in
\Lambda(x)\setminus L(X)$ and this is a contradiction.

\item[(iv)] It can be proved in a similar way.

\item[(v)] It is a consequence of previous results and  Lemma \ref{attraction} (iv).
\end{enumerate}
\end{proof}

\begin{theorem}\label{regionattraction1} Suppose that $X=(X,\phi, \varepsilon(X))$ is an exterior
discrete semi-flow and  $(X,\phi)$ is Lagrange stable at every
point in $\text{PA}(\bar L (X))$. Then,
\begin{itemize}
\item[(i)]  $\text{PA}(\bar L (X)) \setminus \text{WA} (\bar L(X)\setminus L(X)) \subset \bar D(X)$,
\item[(ii)] If,  in addition,  $\text{WA}(\bar L(X)\setminus L(X))=\emptyset$, one has that
$$\text{A}(\bar L (X)) = \bar D(X).$$
\end{itemize}
\end{theorem}

\begin{proof}

\begin{enumerate}
\item[(i)] Suppose that $x\in \text{ \it PA}(\bar L (X)) \setminus \text{\it WA} (\bar L(X)\setminus L(X))$
and $x\notin \bar D(X)$. Then, there is $E\in \varepsilon(X)$
and an increasing sequence $n_1<n_2< \cdots$ such that
$\phi^{n_k}(x) \in X\setminus \bar E$. Since $X$ is Lagrange
stable at $x$, there is a subnet $n_{k_i}$ such that $n_{k_i}
\cdot x \to y \in X \setminus E$. Then, $y \in \Lambda(x)\setminus
L(X)$  and $x\in  \text{\it WA} (\bar L(X)\setminus L(X)),$  which is
a contradiction.

\item[(ii)] It follows by (i), Theorem \ref{regionattraction} (i) and Lemma \ref{attraction}.
\end{enumerate}
\end{proof}

\section{Basins  associated to end points in a region of attraction}

We note that, for an exterior discrete semi-flow $X=(X,\phi,
\varepsilon(X))$, each trajectory of a point of $D(X)$ has an end
point given as follows.

If $x \in D(X)$, we have the exterior map $\phi_{x} \colon \N \to
X$, $\phi_{x}(n)=\phi(n,x)$, which determines an end point
$[\phi_{x}]\in \pi_0^{\BG}(X)$.

Then, the following canonical map is obtained
$$\omega \colon D(X) \to \pi_0^{\BG}(X),$$
where $\omega(x)=[\phi_{x}]$.

Recall the map $\eta_X \colon \pi_0^S(X) \to \check \pi_0(X)$
given by $\eta_X([\alpha])= \check \pi_0(\alpha) (+\infty)$.

\begin{definition} \label{representable} Let $X=(X,\phi, \varepsilon(X))$ be an exterior discrete semi-flow.
\begin{itemize}
\item  An end
$a\in \pi^{\BG}_{0}(X)$ is said to be {\bf $\omega$-representable}
if there is $x\in D(X)$ such that  $\omega(x)=a$. Denote by
${^{\omega}\pi_0^{\BG}}(X)$ the set of $\omega$-representable
ends of $X$.
\item An end $a\in \F(X)$ is said to be {\bf $ \check \omega$-representable} if
there is $x\in D(X)$ such that, for every $E \in \varepsilon(X)$,
there is $n_E \in \N$ satisfying that, for every $n \geq n_E$,
$\phi^n(x) \in C_E$, where $C_E$ is the path component of $a$ in
$E$. Denote by ${^{\omega}\F}(X)$ the set of $\check
\omega$-representable  ends of $\F(X)$ and by $\check D(X)$  the
set of points $x\in D(X)$  such that there is $a\in \F(X)$
satisfying that, for every $E\in \varepsilon(X)$, there is $n_E$
such that, for every $n\geq n_E$, $\phi^n(x) \in C_E$, where $C_E$
is the component of $a$ in $E$.
\item An end $a=[\alpha]\in \pi_0^S(X)$, with $\alpha \colon \R_+ \to X$ an
exterior map, is said to be {\bf $ \check \omega$-$
\eta_X$-representable} if $\eta_X(a)$ is {\bf $\check
\omega$-representable}. Denote by ${^{\omega}\pi_0^S}(X)$ the set
of $ \check \omega$-$ \eta_X$-representable  ends of $\pi_0^S(X)$
and by $D^S(X)$ the set of points $x\in D(X)$ such that there is
$a=[\alpha]\in \pi_0^S(X)$ such that $[\alpha|_\N]=[\phi_x]$.
\end{itemize}
\end{definition}

From the definition of $\check D(X)$, it easy to see that there is an induced map
$$\check \omega \colon \check D(X) \to \check \pi_0(X)$$
which maps $x \in \check D(X) $ to the unique end point $a\in
\check \pi_0(X)$ such that, for every $E\in \varepsilon(X)$, there
is $n_E$ satisfying that, for every $n\geq n_E$, $\phi^n(x) \in
C_E$, where $C_E$ is the component of $a$ in $E$.

\begin{remark} Let $X=(X,\phi, \varepsilon(X))$ be an  exterior discrete semi-flow which is  first-countable at infinity.
In general, the diagram
$$\xymatrix{L(X) \ar[r]^{e} \ar[d] & \F(X)\ar[d]^{\theta_X}\\
D(X)\ar[r]^{\\\omega}  & {\pi_0^{\BG}}(X)}$$ is not commutative.
Nevertheless, the diagram commutes  if we have some  additional
conditions: ``for every $x\in L(X)$ and for every $E \in
\varepsilon(X)$, the points $x, \phi^1 (x)$ are in the same path
component of $E$'' (for instance, when $L(X)\subset \Fix(X)$).
Note that, under this condition,  $e$ agrees with $\check \omega$
on $\check D (X)\cap L(X)$ and the restriction to $\check D
(X)\cap L(X)$ gives a commutative diagram.
   \end{remark}

Recall that the maps $\pi_0^S(\phi^1)$,  $\F(\phi^1)$ and
$\pi_0^{\BG}(\phi^1)$ induce on $\pi_0^S(X)$,  $\F(X)$ and
$\pi_0^{\BG}(X)$  the structure of a discrete semi-flow (set),
respectively. In a similar way, ${^{\omega}\pi_0^S}(\phi^1)$,
${^{\omega}\F}(\phi^1)$ and   ${^{\omega}\pi_0^{\BG}}(\phi^1)$
induce on ${^{\omega}\pi_0^S}(X)$,  ${^{\omega}\F}(X)$ and
${^{\omega}\pi_0^{\BG}}(X)$  the structure of a discrete semi-flow
(set), respectively.

In the following proposition, the  analogues of the results
%the last part of the
%exact sequence given at the end
 of section \ref{limitsandnt} for
representable ends are established:

\begin{proposition}  Let $X=(X,\phi, \varepsilon(X))$ be an exterior discrete semi-flow
and consider the restrictions of the natural transformations $R_X, \eta_X$ to the subsets of $\omega$-representable ends.
\begin{itemize}
\item[(i)]  In the following diagram
$$\xymatrix{{^\omega\pi}_0^{\rm S}(X)\ar[rr]^{{^\omega R}_X}  & & {^\omega\pi}_0^{\BG}(X) \ar@<1ex>[r]^{\Id} \ar@<-1ex>[r]_{{^\omega S}} & {^\omega\pi}_0^{\BG}(X),
}$$
the image of $^\omega R$ is the equalizer of $\Id$ and $ {^\omega S}$.

\item[(ii)] If $X$ is first-countable at infinity, the following diagram is commutative
$$\xymatrix{{^\omega \pi_0}^{\rm S}(X)\ar[rr]^{{^\omega R}_X} \ar@{>>}[rd]_{{^\omega\eta}_X} & &{^\omega \pi}_0^{\BG}(X)  \\
&{^\omega\F}(X) \ar@{>->}[ru]_{{^\omega\theta}_X}}$$
where $ {^\omega R}_X= {^\omega\theta}_X\circ{^\omega\eta}_X $ is a canonical
factorization of $ {^\omega R}_X$ as the composition of an epimorphism and a monomorphism.
In this case, ${^\omega\F}(X)$ is the equalizer of  $\Id$ and $ {^\omega S}$.

\end{itemize}
\end{proposition}

\medskip

Now we study some relations between the different subsets of pseudo-attraction associated to an externology given in Definition \ref{representable} and the different types of representable end points.

\begin{proposition} \label{shift} Suppose that $X=(X,\phi, \varepsilon(X))$ is an exterior discrete semi-flow. Then:
\begin{itemize}
\item[(i)] There are natural inclusions
$D^S(X) \subset \check D(X) \subset D(X)$, $\check D(X)$ is completely invariant and $D^S(X)$ is right-invariant.
\item[(ii)] The maps $\omega  \colon D(X) \to \pi_0^{\BG}(X)$, $\check\omega  \colon \check D(X) \to \check\pi_0(X)$ are
discrete semi-flow morphisms (between discrete semi-flow sets) which induce discrete
semi-flow epimorphisms  $\omega  \colon D(X) \to
{^{\omega}\pi_0^{\BG}}(X)$, $\check\omega  \colon \check D(X) \to
{^{\omega}\F}(X)$.
\item[(iii)] If $\pi_0(\phi^1)$ is injective, then ${^{\omega}\pi_0^{\BG}}(\phi^1)$ is injective.
If, in addition, $\phi^1$ is surjective, then
${^{\omega}\pi_0^{\BG}}(X)$  has the structure of a discrete flow
set.
\item[(iv)] The action induced on ${^{\omega}\F}(X)$ is trivial, that is, ${^{\omega}\F}(X) \subset \mathrm{Fix}(\F(X))$.
 \item [(v)] $S\,  \pi_0^{\BG}(\phi^1)= \pi_0^{\BG}(\phi^1)\, S$,
 ${^{\omega}\pi_0^{\BG}}(X)\subset \{a\in \pi_0^{\BG}(X)\mid S(a)=\pi_0^{\BG}(\phi^1)(a)\}$.
\end{itemize}
\end{proposition}

\begin{proof}
\begin{enumerate}
\item[(i)] It is straightforward to check.

\item[(ii)] Given $x\in D(X) $, we have
$\phi_{\phi^1(x)}(n)=\phi^1 \phi_{x}(n)$. Then,
$\omega(\phi^1(x))=[\phi_{\phi^1(x)}]=[\phi^1\phi_{x}]=\pi_0^{\BG}(\phi^1)[\phi_{x}]= \pi_0^{\BG}(\phi^1) \omega (x)$.
There is a similar proof for $\check \omega$.

\item[(iii)]
Suppose that
$\pi_0^{\BG}(\phi^1)[\phi_{x}]=\pi_0^{\BG}(\phi^1)[\phi_{x'}]$.
Then, we have that $\phi^1(x)$ is in the same path component that
$\phi^1(x')$. Taking into account that $\pi_0(\phi^1)$ is
injective, we have that there is a continuous path from $x$ to
$x'$. Using this path and an exterior homotopy from
$\phi^1_{\phi^1({x})}$ to $\phi^1_{\phi^1({x'})}$, we can
construct a new exterior homotopy from $\phi_{x}$ to
$\phi_{x'}$.

Now suppose that  $\phi^1$ is surjective. Then, given $y\in D(X)$, we can take $x\in X$ such that $\phi^1(x)=y$.
Then, $x\in D(X)$ and one has that $\pi_0^{\BG}(\phi^1)[\phi_x]=[\phi_y]$. This implies that $\pi_0^{\BG}(\phi^1)$ is surjective.

\item[(iv)] Just considering that $\check \omega  (\phi^1(x))= \check \omega  (x)$, $x \in \check D(X)$, it follows that the action is trivial.

\item[(v)] If $[\alpha] \in \pi_0^{\BG}(X)$, then
$S \pi_0^{\BG}(\phi^1)[\alpha] =[(\phi^1 \alpha) s]=[\phi^1(
\alpha s) ]=\pi_0^{\BG}(\phi^1) S [\alpha]$. Besides, if $x\in
D(X)$, then $S\omega(x)=
S[\phi_{x}]=[\phi^1\phi_{x}]=\pi_0^{\BG}(\phi^1)\omega(x)$.
\end{enumerate}
\end{proof}

Now we introduce the basins of  end points and we consider special  end points whose basins are open subset of $X$. Some local stability notions with respect to end points are
also considered.

\begin{definition} Given an exterior discrete semi-flow  $X=(X,\phi, \varepsilon(X))$, the subspace denoted by
 $$D_{a}=\omega^{-1}(a), \quad a \in {^{\omega}\pi_0^{\BG}}(X)$$
will be called the {\bf $\omega$-basin of  $a$}.

The subspace  $$\check D_{a}=\check \omega^{-1}(a), \quad a \in  {^{\omega}{\check \pi}}_0(X)$$
will be called the {\bf $\check\omega$-basin of  $a$}.

The subspace  $$D_{a}^S=D^S(X) \cap \check D_{\eta_X(a)}, \quad a
\in  {^{\omega}{\pi}}_0^S(X)$$ will be called the {\bf immediate
basin of  $a$}.
\end{definition}

\begin{definition} Let $X$ be an exterior discrete semi-flow.
\begin{itemize}
\item An end point $ a \in  {^{\omega}\pi_0^{\BG}}(X)$ is said to be an {\bf $\omega$-attractor} if its basin
$D_a$ is an open subset of $X$. Denote
${^{\omega}_A\pi_0^{\BG}}(X)=\{a\in{^{\omega}\pi_0^{\BG}}(X)\mid a
\mbox{ is an $\omega$-attractor}\}$ and
${^{\omega}_{N}\pi_0^{\BG}}(X)={^{\omega}\pi_0^{\BG}}(X) \setminus
{^{\omega}_A\pi_0^{\BG}}(X)$.

\item  An end point $ a \in  {^{\omega}{\check \pi}}_0(X)$ is said to be an {\bf $\check\omega$-attractor} if its basin
$ \check D_a$ is an open subset of $X$.
Denote ${^{\omega}_A{\check \pi}}_0(X)=\{a \in  {^{\omega}{\check \pi}}_0(X)\mid a \mbox{ is a $\check\omega$-attractor}\}$ and
 ${^{\omega}_N{\check \pi}}_0(X)= {^{\omega}{\check \pi}}_0(X) \setminus {^{\omega}_A{\check \pi}}_0(X)$.

\item A point $x\in D(X)$ is {\bf locally  $\omega$-stable} if there is an open neighborhood $U$ such that
$U\subset D(X)$ and, for every $y \in U$, $\omega(y)=\omega(x)$.

\item A point $x\in \check D(X)$ is {\bf locally  $\check\omega$-stable} if there is an open neighborhood $U$
such that $U\subset \check D(X)$ and,  for every $y \in U$,
$\check \omega(y)= \check \omega(x)$.
\end{itemize}
\end{definition}

  We have the following basic properties which are immediately deduced from the definitions.

\begin{lemma}  Let $X=(X,\phi, \varepsilon(X))$ be an exterior discrete semi-flow and let $x\in D(X)$.
The following statements hold:
\begin{itemize}
\item[(i)] If the end $\omega(x)$ is an $\omega$-attractor, then $x$ is locally $\omega$-stable.
\item[(ii)] If $a\in {^{\omega}\pi_0^{\BG}}(X)$ and $x\in D_a$, $x$ is locally $\omega$-stable
if and only if $x \in \mathring{D}_a$. If, for every $x \in D_a$,
$x$ is locally $\omega$-stable, then $a$ is an $\omega$-attractor.
\item[(iii)] If the end $\check \omega(x)$ is an $\check \omega$-attractor, then $x$ is locally $\check \omega$-stable.
\item[(iv)] If $a \in  {^{\omega}{\check \pi}}_0(X)$ and $x\in \check D_a$, $x$ is locally
$\check\omega$-stable if and only if  $x \in \mathring{\check
D}_a$. If, for every $x \in \check D_a$, $x$ is locally
$\check\omega$-stable, then  $a$ is a $\check\omega$-attractor.
\end{itemize}
\end{lemma}

The maps   $\omega, \check \omega$ permit us  to divide an exterior discrete semi-flow.

\begin{corollary} Let $X=(X,\phi, \varepsilon(X))$ be an exterior discrete semi-flow
and denote $D=D(X)$, $\check D=\check D(X)$, $D^S=D^S(X)$. Then, we have
the following induced partitions of $X$:
    $$X=(X\setminus D)\sqcup \left ( \bigsqcup_{a \in {^{\omega}\pi_0^{\BG}}(X)}D_{a}\right),$$
     $$X=(X\setminus \check D)\sqcup \left ( \bigsqcup_{a \in {^{\omega}{\check \pi}}_0(X)} \check D_{a}\right),$$
        $$X=(X\setminus D^S)\sqcup \left ( \bigsqcup_{a \in {^{\omega}\pi}^S_0(X)} D^S_{a}\right),$$
where $X\setminus D$,  $X\setminus \check D$ and $\check D_a$  are completely invariant and  $D^S_{a}$ is right-invariant.
\end{corollary}

In general $D_a$ is not right invariant.
However, some notable unions of these basins are completely invariant subsets:

\begin{proposition} Let $X=(X,\phi, \varepsilon(X))$ be an exterior discrete semi-flow. Then:
\begin{itemize}
\item[(i)]
If $a \in {^{\omega}\pi_0^{\BG}}(X)$ and $[a]$ is  the completely invariant big orbit  of
$a$ in ${^{\omega}\pi_0^{\BG}}(X)$, then $\bigsqcup_{a' \in [a]}D_{a'}$ is completely invariant.
\item[(ii)] If $\pi_0(\phi^1)$ is injective and $\phi^1 \colon X \to X$ is a surjective open map,
then ${_A^{\omega}\pi_0^{\BG}}(X)$ is completely invariant in ${^{\omega}\pi_0^{\BG}}(X)$,
 $\bigsqcup_{a \in {_A^{\omega}\pi_0^{\BG}}(X)}D_{a}$ is an open completely invariant subset of
 $X$ and $\bigsqcup_{a \in {_N^{\omega}\pi_0^{\BG}}(X)}D_{a}$ is completely invariant and closed in $D(X)$.
 \item[(iii)] If $X$ is first-countable at infinity, then the following diagram is commutative
$$\xymatrix{\check D(X) \ar[r]^{\check \omega} \ar[d] & {^{\omega}\F}(X)\ar[d]^{\theta_X}\\
D(X)\ar[r]^{\\\omega}  & {^{\omega}\pi_0^{\BG}}(X)}$$

\begin{proof} A simple inspection proves (i) and (iii).

It only remains to check  (ii). By Proposition \ref{shift}
(ii)-(iii), we have a commutative diagram

$$\xymatrix{D(X)\ar[r]^\omega\ar[d]^{\phi^1} & {^\omega}\pi_0^\BG(X)\ar[d]^{{^\omega}\pi_0^\BG(\phi^1)}\\
D(X)\ar[r]_\omega & {^\omega}\pi_0^\BG(X)}$$
   where ${^{\omega}\pi_0^{\BG}}(\phi^1)$ is  bijective. This implies that
   $\omega^{-1}({^{\omega}\pi_0^{\BG}}(\phi^1)(a))=\phi^1 (\omega^{-1}(a))$, for $a \in {^{\omega}\pi_0^{\BG}}(X)$.
   Since $\phi^1$ is a continuous open  map, it follows that ${_A^{\omega}\pi_0^{\BG}}(X)$ is completely invariant
   and also the rest of the assertion.
    \end{proof}

 \end{itemize}
\end{proposition}

%Notice that the inclusion $D(X)\subset X$ of exterior spaces
%induces the transformation $L(D(X)) \to L(X)$.
%
%\begin{proposition}\label{eledeigualele} Let $X=(X,\phi, \varepsilon(X))$ be an exterior discrete semi-flow.
%Then, $L(D(X))= L(X)$ (similarly, $\bar L(\bar D(X))= \bar L(X)$).
%\end{proposition}

Using the subset of end points which are attractors and its complement,
one can divide, under the conditions of Proposition above,  the region of pseudo-attraction of an externology as the
union of a completely invariant open subset and its complement:

\begin{definition} Given an exterior discrete semi-flow $X=(X,\phi, \varepsilon(X))$,
$${_A D}(X)=\bigsqcup_{a \in {_A^{\omega}\pi_0^{\BG}}(X)}D_{a}, \quad
{_N D}(X)=\bigsqcup_{a \in {_N^{\omega}\pi_0^{\BG}}(X)}D_{a}$$
are   the {\bf ${^{\omega}\pi_0^{\BG}}$-attracting basin} and the
{\bf ${^{\omega}\pi_0^{\BG}}$-non-attracting basin} of $X$, respectively.
Similarly,
$${_A \check D}(X)=\bigsqcup_{a \in {_A^{\omega}{\check \pi}}_0(X)} \check D_{a}, \quad
{_N \check D}(X)=\bigsqcup_{a \in {_N^{\omega}{\check \pi}}_0(X)} \check D_{a}$$
 are   the {\bf ${^{\omega}\check \pi_0}$-attracting basin} and the {\bf ${^{\omega}\check \pi_0}$-non-attracting basin} of $X$.

The intersection with the limit sets
 $${_A L}(X)={_A D}(X) \cap L(X), \quad
 {_N L}(X)={_N D}(X) \cap L(X)$$
  $${_A \check L}(X)={_A \check D}(X) \cap L(X), \quad
 {_N \check L}(X)={_N \check  D}(X) \cap L(X)$$

are the corresponding {\bf attracting (non-attracting) limit} of $X$.

\end{definition}

\begin{remark}

In a similar way, we can consider
canonical maps
$${\bar\omega} \colon {\bar D}(X) \to {\bar\pi}_0^{\BG}(X),$$
and basins of the form $${\bar D}_{a}={\bar\omega}^{-1}(a), \quad a \in {^{\omega}{\bar \pi}_0^{\BG}}(X)$$
The new sets end points and bar-limits can be divided with the corresponding notions of attracting end points. The authors suggest  that subsets of the following type
$$%{_A \bar L}(X)={_A \bar D}(X) \cap \bar L(X), \quad
{_N \bar L}(X)={_N \bar D}(X) \cap \bar L(X)$$
$$%{_A \check  {\bar L}}(X)={_A \check { \bar D}}(X) \cap \bar L(X), \quad
{_N \check { \bar L}}(X)={_N\check  {\bar D}}(X) \cap \bar L(X)$$
can be used to give pure topological  analogues of Julia  sets in a topological semi-flow.

\end{remark}

\section{The externology of right-absorbing open subsets}

For a given $X=(X,\phi)$   discrete semi-flow, we can consider the
externology $\varepsilon^{\r}(X)$ given by all the open subsets
$E$ such that, for every $x\in X$, there is $n\in \N$ satisfying
that, for $m\geq n$, $\phi^m(x) \in E$. For this externology, one
has that  $(X, \phi, \varepsilon^{\r}(X))$ is a
$\d$-exterior discrete semi-flow, then $D(X)=X$.

\begin{definition} For a given $(X,\phi)$ discrete semi-flow, the externology $\varepsilon^{\r}(X)$
is said to be the {\bf right-absorbing externology} of $(X, \phi)$
(or just the {\bf right externology}) and an open $E\in
\varepsilon^{\r}(X)$ is said to be an $\r$-exterior open subset of
$X$.
\end{definition}

In this section, we  consider the following basic properties.

\begin{lemma}\label{exteriorandD} Let $(X, \phi, \varepsilon(X))$ be an exterior discrete semi-flow.
Then, the following statements are equivalent:
\begin{itemize}
\item[(i)] $\varepsilon(X) \subset \varepsilon^{\r}(X)$.
\item[(ii)] $(X, \phi, \varepsilon(X))$ is a $\d$-exterior discrete semi-flow.
\item[(iii)] $D(X)=X$.
\end{itemize}
\end{lemma}

It is important to note that one has a canonical functor
$\bf{F}(\N) \to \bf{E^{\d}F}(\N)$ which carries $(X, \phi)$ to
$(X, \phi, \varepsilon^{\r}(X))$. We will use the reduced notation
$X=(X, \phi)$ and $X^{\r}= (X, \phi, \varepsilon^{\r}(X))$.  Using
this canonical construction  all the different
constructions (limits, end sets, etc.) given for exterior discrete
semi-flows can be applied to a discrete semi-flow.

In  next subsections, some relations between constructions associated to an
exterior space and dynamic properties of a discrete semi-flow are analyzed. In particular,
we see the relation between limits and periodic points, as well as between bar-limits and omega-limits.

\subsection{Periodic points}

The relation of the limit space of an exterior discrete semi-flow
and the sub-flow of periodic points is analyzed in the following
results.

\begin{lemma}\label{uno} If $X$ is a $\d$-exterior discrete semi-flow, then
$P(X)\subset L(X)$. In particular, if $X$ is a discrete semi-flow, then
$P(X)\subset L(X^{\r}).$
\end{lemma}
\medskip
\begin{proof} Take a periodic point $x$ and an arbitrary $E\in \varepsilon(X)$.
Then, there exists $n \in \N$ such that $(n,+\infty)\cdot x \subset E.$ Since
$x$ is periodic, $(n,+\infty)\cdot x= \N \cdot x$ and, taking into account
that $x\in \N \cdot x,$ we have that $x\in E.$
\end{proof}

\begin{lemma}\label{dos}
Let $X$ be a discrete semi-flow and suppose that $X$ is a $T_1$-space. Then, for
every $x\in X$,  the following statements are equivalent:
 \begin{itemize}
  \item[(i)] $x$ is a non-periodic point.
  \item[(ii)]  $X\setminus \{x\}$ is  an $\r$-exterior open subset of $X.$
\end{itemize}
\end{lemma}
\medskip
\begin{proof} In order to prove that (i) implies (ii), take $y\in X$; if $(\N \cdot y) \cap (\N \cdot x )= \emptyset $,  then, for every $n\in
\N,$ $(n,+\infty) \cdot y \subset X\setminus \{x\}.$ If $(\N \cdot y) \cap (\N \cdot x) \not = \emptyset $, considering that $x$ is not periodic, one can
find $n\in \N$ such that $(n,+\infty)\cdot y \subset X\setminus \{x\}.$
Then, one has that $X\setminus \{x\} \in \varepsilon^{\r} (X).$
Conversely, suppose that $x$ is a periodic point. Then, by Lemma
\ref{uno} above, $X\setminus \{x\}$ is not $\r$-exterior.
\end{proof}

Using these two lemmas, we obtain the following result.

\begin{theorem}\label{periodic} Let $X$ be a discrete semi-flow and suppose that $X$ is a
$T_1$-space. Then $$P(X)=L(X^{\r}).$$
\end{theorem}
\medskip
\begin{proof} Let $x\in X\setminus P(X)$.
Then, by Lemma \ref{dos}, one has that $X\setminus \{x\} \in
\varepsilon^{\r} (X)$ and $$P(X)=X \setminus \left(\bigcup_{x \notin
P(X)}\{x\}\right)=\bigcap_{x \notin P(X)} X\setminus \{x\}\supset
\bigcap_{E\in \varepsilon^{\r}(X)}E=L(X^{\r}).$$ Now, the result
follows from Lemma \ref{uno}.
\end{proof}

Taking into account the theorem above, if $X$ is a $T_1$ discrete semi-flow, in the diagram of distinguished
sub-flows  $$P(X)\subset{\rm Poisson}(X)\subset \Lambda (X^{\r})$$ we have identify the first sub-flow as a limit set:
$$L(X^{\r})=P(X)\subset{\rm Poisson}(X)\subset \Lambda (X^{\r}) \subset X.$$

\subsection{Limits and omega-limits}

In the following result, we analyze the relationship between the
omega-limit and the bar-limit induced by an externology.

\begin{lemma}\label{omega_bar_limit} If $X$ is an exterior discrete semi-flow, then
$$\Lambda (D(X))\subset \Lambda (\bar D(X)) \subset \overline{\Lambda (\bar D(X))} \subset \bar L (X).$$
\end{lemma}
\medskip
\begin{proof} If $E \in \varepsilon (X)$, then, for every $x \in \bar D(X)$, there
exists $n\in \N$ such that $(n,+\infty)\cdot x \subset \overline{E}$ and therefore
$\overline{(n,+\infty)\cdot x} \subset \overline{E}.$ By definition, this
implies that $\Lambda (x) \subset \bar L (X)$, for every $x \in \bar D(X)$.
Hence, $\Lambda ( \bar D(X)) \subset \bar L(X).$ Taking into account that
$\bar L (X)$ is a closed subset, we also have
$\overline{\Lambda (\bar D(X))} \subset \bar L (X)$.
\end{proof}

And now, let us present some technical results.

\begin{lemma}\label{exterior_bar_limit}
Let $(X,\varepsilon(X))$ be an exterior space  and
$x\in X.$ Then, there exists an open neighborhood $V_x$ at $x$ such that
$X\setminus \overline{V}_x \in \varepsilon(X)$ if and only if $x
\notin \bar L (X)$.
\end{lemma}
\medskip
\begin{proof} If $X\setminus \overline{V}_x  \in \varepsilon(X)$,
then, taking into account that $V_x \cap (X\setminus
\overline{V}_x )= \emptyset $, we have that $x \notin
\overline{X\setminus \overline{V}_x}$. Since $X\setminus
\overline{V}_x\in \varepsilon(X)$, it follows that $x\notin
\bigcap_{E \in \varepsilon(X)} \overline{E}=\bar L (X)$.

Conversely, if $x \notin \bar L (X),$ then there exists $E \in
\varepsilon(X)$ such that $x\notin \overline{E}.$  Now, taking $V_x=
X\setminus \overline{E}=\mbox{Int}(X\setminus E)$, we have that
$X\setminus \overline{V}_x=\mbox{Int}(X\setminus
V_x)=\mbox{Int}(\overline{E}) \supset E.$ Consequently,
$X\setminus \overline{V}_x\in \varepsilon(X)$.
\end{proof}

\begin{lemma}\label{noper}
Let $X$ be an exterior discrete semi-flow, let $\varepsilon(X)$ be its externology and let
$x\in X.$ If there exists an open neighborhood $V_x$ at $x$ such that
$X\setminus \overline{V}_x\in \varepsilon(X)$, then $x \notin
\overline{\Lambda ( \bar D(X))}.$
\end{lemma}

\begin{proof} It is a consequence of Lemma \ref{exterior_bar_limit}
and  Lemma \ref{omega_bar_limit}.
\end{proof}

\begin{proposition}
Let $X$ be a discrete semi-flow  and $x\in X.$ If there exists
an open neighborhood $V_x$ at $x$  such that $X\setminus \overline{V}_x$ is
$\r$-exterior, then $x\notin \overline{\Lambda (X)}.$
\end{proposition}

\begin{proof} It is a consequence of Lemma \ref{noper}
and  Lemma \ref{exteriorandD}.
\end{proof}

\begin{lemma} Let $X$ be a discrete semi-flow which is  a locally compact regular space.
If  $x \notin\overline{ \Lambda (X)},$ then there exists an open
neighborhood $V_x$ at $x$ such that $X\setminus \overline{V}_x$ is
$\r$-exterior.
\end{lemma}

\begin{proof} Suppose that $x \notin\overline{ \Lambda (X)}.$
Since $X$ is locally compact, there is a compact neighborhood $K$
at $x$ such that $K \cap \Lambda (X)= \emptyset.$ Take $y\in X$ and
assume that, for every $m \in \N$, $(m,+\infty)\cdot y \cap K \not =
\emptyset .$ Then, there is a sequence $m_n\rightarrow +\infty $
such that $m_{n}\cdot y\in K.$ Being $K$ compact, one can take a
subnet $m_{n_i}\rightarrow +\infty$ such that $m_{n_i}\cdot y
\to u \in K.$ This fact implies that $u\in K \cap \Lambda (y)
\subset K\cap \Lambda (X),$ which is a contradiction. Therefore,
there is $m\in \N$ such that $(m,+\infty) \cdot y \cap K= \emptyset.$ By the
regularity of $X$, there exists an open neighborhood $V_x$ at $x$ such that
$\overline{V}_x \subset K$ and $X\setminus \overline{V}_x$ is
$\r$-exterior.
\end{proof}

\begin{corollary} Let $X$ be a discrete semi-flow. If $X$ is a locally compact regular space,
then $\bar L(X^{\r}) \subset \overline{\Lambda (X)}.$
\end{corollary}

\begin{proof} If $x \notin \overline{\Lambda (X)}$, by the lemma above there exists an open
neighborhood $V_x$ at $x$  such that $X\setminus \overline{V}_x$
is  $\r$-exterior. By Lemma \ref{exterior_bar_limit}, it follows
that $x \notin \bar L(X^{\r}).$
\end{proof}

By the corollary above and Lemma \ref{omega_bar_limit}, we obtain
the following result.

\begin{theorem} \label{main}Let $X$ be a discrete semi-flow. If $X$ is a locally compact
regular space, then $\bar L(X^{\r}) = \overline{\Lambda (X)}.$
\end{theorem}

\begin{corollary} \label{reticulo} Let $X$ be a discrete semi-flow. If $X$ is a locally compact
$T_3$  space, then $$L(X^{\r})=P(X)\subset{\rm Poisson}(X)\subset
\Lambda(X) \subset \overline{\Lambda (X)}=\bar L(X^{\r}).$$
\end{corollary}
\medskip
\begin{proof}
This is a consequence of Theorems \ref{periodic} and \ref{main}.
\end{proof}

\section{Other externologies and examples}\label{eledeigualele}

If we take a right-invariant subset $S$ of a discrete semi-flow
$X$, new families of externologies can be analyzed: for instance,
we can consider the externology $\epsilon(X,S)$ formed by all the
open neighborhoods of $S$ in $X$.

As an application of the notions and constructions developed in
this paper, we analyze the decompositions given by externologies
induced by  the right-invariant subsets of 1-periodic (fixed)
points $P_1$ and 2-periodic points $P_2$.

In the following example, we take on the Riemann sphere $X=\C\cup \{\infty\}$
the discrete semi-flow induced by the polynomial function  $h(z)= z^2-1$ ($h(\infty)=\infty$).

The fixed points of $h$ are $\infty$ and the two golden numbers
$p_{1}=\frac{1-\sqrt{5}}{2}$ and $p_{2}=\frac{1+\sqrt{5}}{2}$. In
this case, we can consider the exterior discrete semi-flow
$X_1=(\C\cup \{\infty\}, h, \epsilon(X, P_{1}))$ and the induced
map
$$\omega_1 \colon D(X_1) \to {^\omega \pi}_0^{\BG}(X_1)$$
from the region of attraction $D(X_1)$ to
the set of omega-representable end points of Brown-Grossman type.
In this case, one has  a canonical isomorphism $P_1\cong {^\omega \pi}_0^{\BG}(X_1)$.

The externology $\epsilon(X, P_{1})$ induces a decomposition of
the following form:
$$X=(X\setminus D(X_{1})) \sqcup \omega_{1}^{-1}(\infty)\sqcup \omega_{1}^{-1}(p_{1})\sqcup\omega_{1}^{-1}(p_{2}),$$
which can be seen in the  Figure \ref{1periodic}.

The basin of $\infty$ is displayed using brown color, the golden
numbers are repulsing points and its basins are not visible in the
figure on the left.  Nevertheless, when we apply a zoom effect, it
is possible to see some points in the basins of the end points
associated to the golden numbers: the central point in the figure
on the right is a point of the basin of a golden number $p_{1}$.
The black color corresponds to points which are not in basins of
fixed points; that is, to $X\setminus D(X_{1})$.

\begin{figure}[h!]
\begin{center}
\includegraphics[scale=0.2]{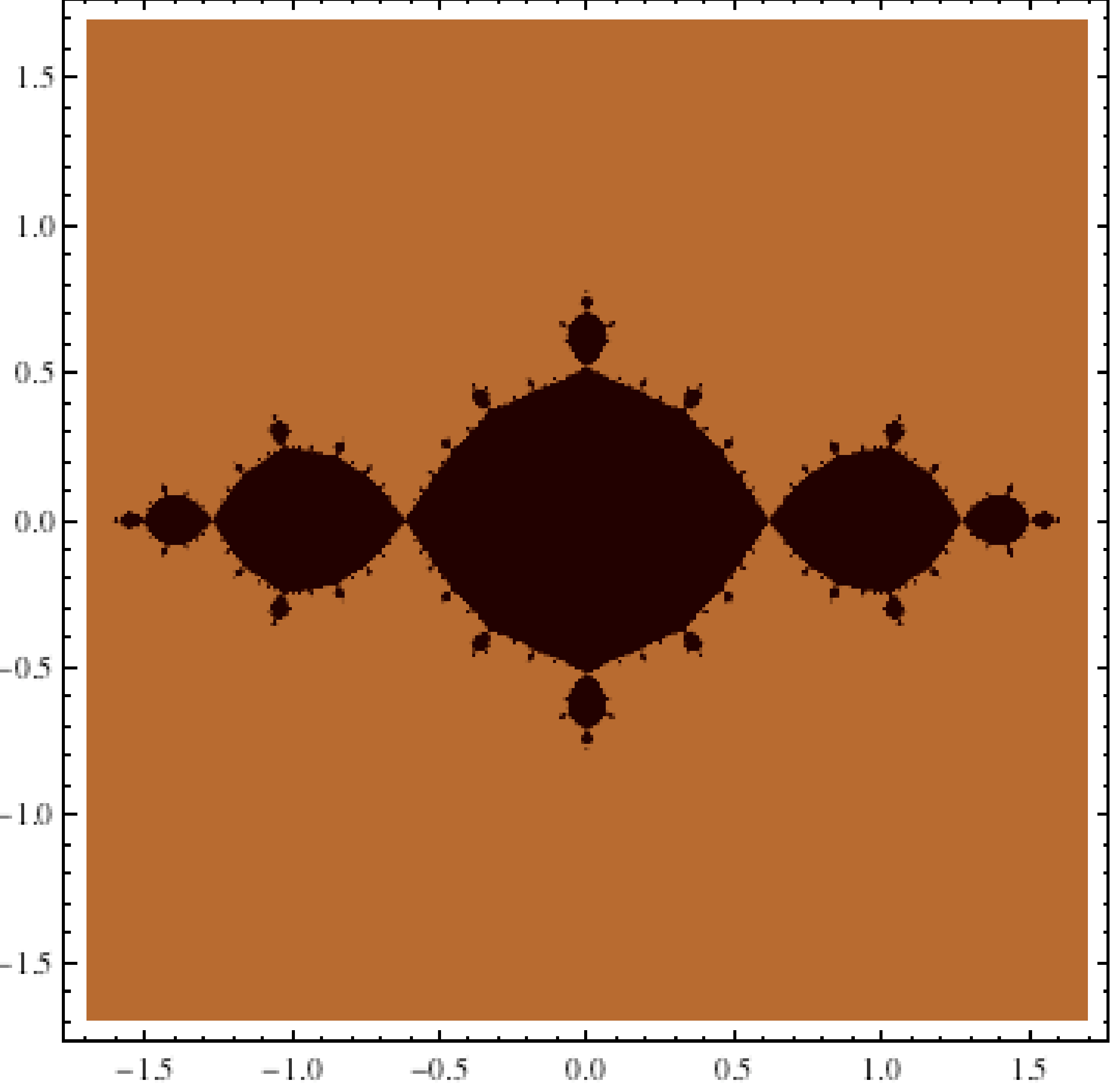} \includegraphics[scale=0.2]{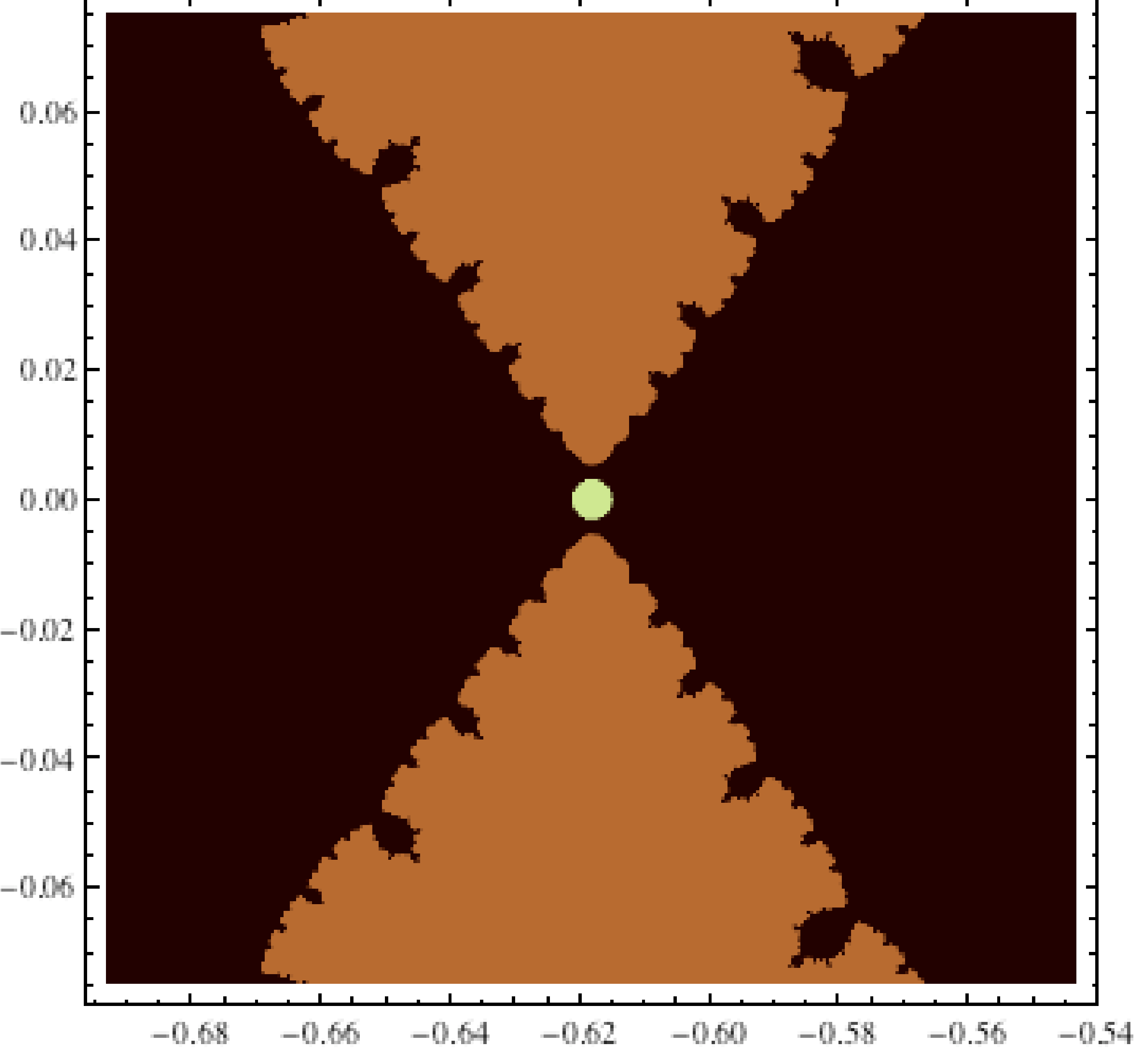}\\
\end{center}
\caption{On the left, the basin of the attracting infinity point (brown) and, on the right, a point of the basin of the repelling fixed point $p_1$.}
\label{1periodic}
\end{figure}

 When the 2-periodic points of $h$ are studied, one obtains:
$ P_2=\{\infty,-1,  p_{1}, 0, p_{2}\}$.
 For the exterior discrete semi-flow $X_2=(\C\cup \{\infty\}, h, \epsilon(X, P_{2}))$ and  the induced map
$$\omega_2 \colon D(X_{2}) \to {^\omega \pi}_0^{\BG}(X_2)$$
from the region of attraction $D(X_2)$ to
the set of omega-representable end points of Brown-Grossman type, one has a canonical isomorphism $P_2\cong {^\omega \pi}_0^{\BG}(X_2)$.
For $X_2$, it is interesting to check that
$P_1\cong {{^\omega\check \pi}}_0(X_2) \cong  {^\omega \pi}_0^{S}(X_2)$.
The cyclic point $-1$
generates the Brown-Grossman end point represented by the sequence
$(-1, 0, -1, 0, \dots)$  and the cyclic point 0 induces the end
point represented by $(0, -1, 0, -1 \dots)$.

To compare the decompositions induced by $P_{1}$ and $P_{2}$, one has to analyse the
diagram

$$ \xymatrix{ D(X_{1}) \ar[r]^{\omega_{1}} \ar[d]&   {^\omega \pi}_0^{\BG}(X_1) \ar[d] \\
   D(X_2) \ar[r]^{\omega_{2}}&   {^\omega \pi}_0^{\BG}(X_2) }$$
where  $D(X_1) \to   D(X_2)$ is the inclusion:
 $$ \omega_{1}^{-1}(\infty)\sqcup \omega_{1}^{-1}(p_{1})\sqcup\omega_{1}^{-1}(p_{2}) \subset
 \omega_{2}^{-1}(\infty)\sqcup \omega_{2}^{-1}(-1)\sqcup\omega_{2}^{-1}(p_{1}) \sqcup\omega_{2}^{-1}(0) \sqcup\omega_{2}^{-1}(p_{2}) $$
It is interesting to note that, in Figure \ref{2periodic}, the
black part associated to $\omega_{1}$ has been divided into the
union of the basins of the new 2-cyclic attracting points of $h$.
The immediate basins $D_{-1}^S$ and $D_{0}^S$ form the path
component of the basin containing the 2-cyclic points $-1$ and
$0$, respectively.
\begin{figure}[h!]
 \begin{center}
\includegraphics[scale=0.2]{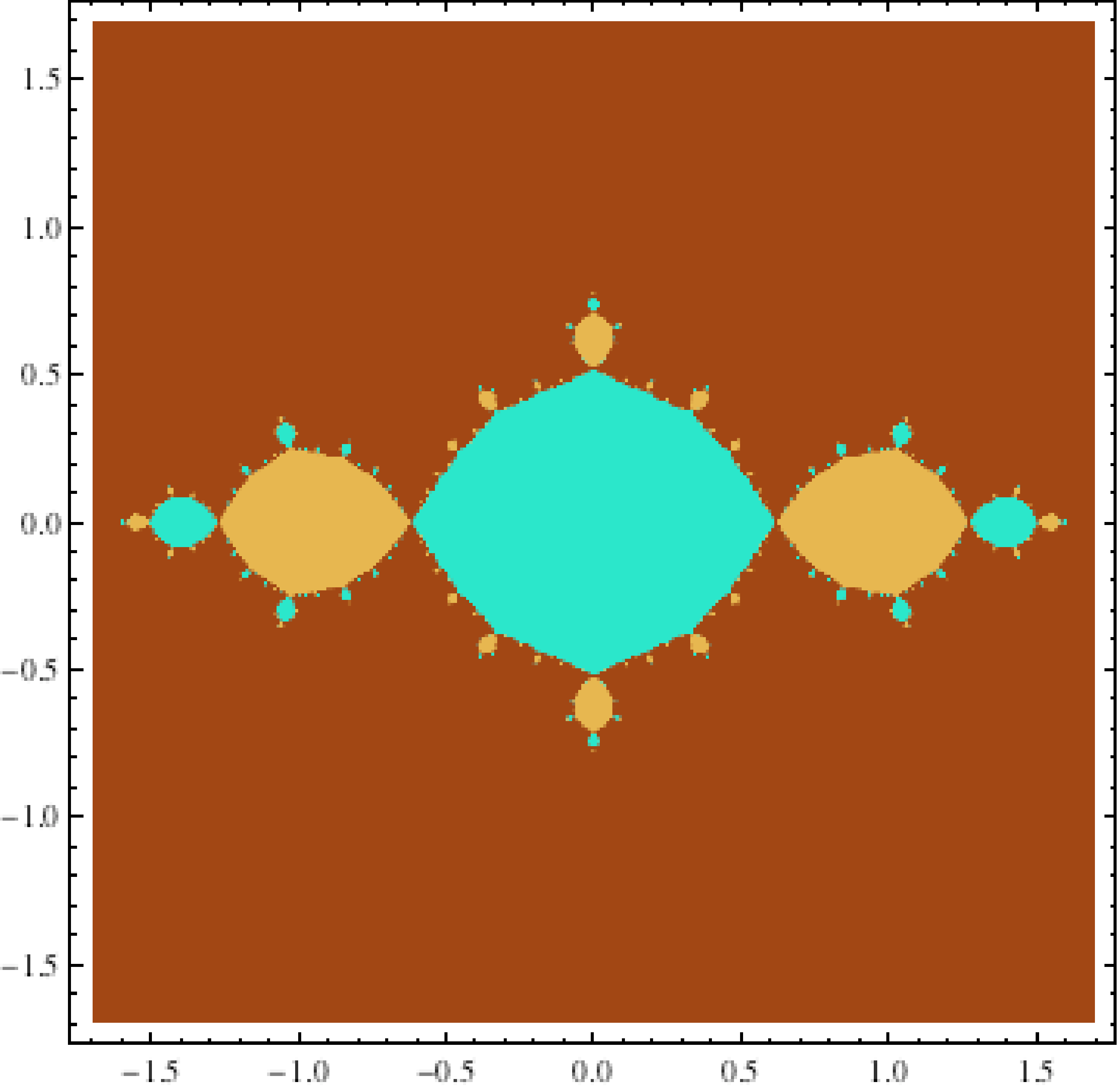}
\end{center}
\caption{The basins of the new 2-cyclic points  correspond to the black part of Figure \ref{2periodic}.}
\label{2periodic}
\end{figure}

\begin{figure}[h!]
\begin{center}
\includegraphics[scale=0.5]{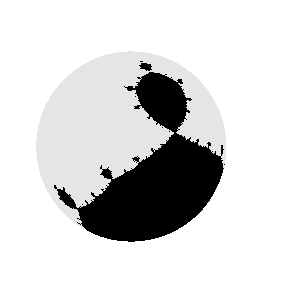}
\includegraphics[scale=0.45]{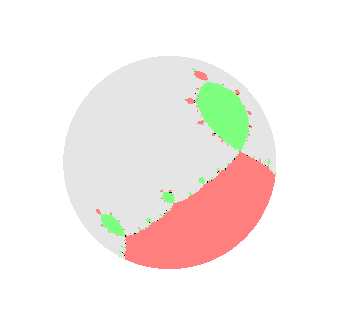}
\end{center}
\caption{The basins of $D(X_1)$ and $D(X_2)$  with new colors on the Riemann sphere.}
\label{sphere}
\end{figure}

\section{Future applications}

The authors think that it will be interesting to develop a deeper
study of externologies associated to a right-invariant subspace
$S$ of a discrete semi-flow $X$ given by all the open
neighborhoods of $S$ in $X$.
%In this sense, some initial results can be seen in \cite{M}.
The results obtained for the example
given in section above suggest that some of the particular
properties analyzed can be proven for more general discrete
semi-flows. For  externologies given by open neighborhoods of an
invariant subset $S$, one can prove that, under good conditions,
the region of attraction of the externology is the region of
attraction of $S$. In some cases, the  limit  of the externology
agrees with $S$ and the set of end points of Borsuk-\v Cech type
is related with the connected components of $S$. Moreover, the
omega-representable  end points of Borsuk-\v Cech type are related
to the fixed points of $S$, and the basins of end points, under
suitable conditions, correspond to basins of fixed points. The
basin of a representable  end point of Steenrod type is related to
the immediate component of the basin of a fixed point. The basins
of omega-representable end points of Brown-Grossman type are
related to the basins of periodic points which are contained in
$S$.

It is worth pointing out that the open neighborhoods of a subset
$S$ on a manifold can also be taken as a resolution (in the sense
of shape theory) of $S$ and this fact gives the  possibility of
applying  techniques of shape theory, like shape Conley index (see
\cite{Salamon}), algebraic characterization of shape equivalences,
et cetera. We can also use these externologies/resolutions to
compare techniques and results coming from   exterior homotopy
theory and from shape theory. For instance, the  exterior
spaces can also be used for the study of the exterior homotopy
type  of  attractors and repellers of discrete and continuous
flows  and to compare them with their corresponding shapes, see
\cite{MR, MSS, Gabites, Sanjurjo11}.

The study of semi-flows  of the form $(M, h)$, where $M$ is a
manifold and $h\colon M \to M$ is a branched covering, can be also
analyzed with the techniques developed in this paper, by either
taking externologies given by open neighborhoods of suitable
subsets of periodic points or taking the right-absorbing
externology. Many questions arise, such as the study of
Brown-Grossman end points associated to points in  the singular
subset of the branched covering. A particular case of this study
is the dynamics of  a rational map on the Riemann sphere and the
study of the structure of the corresponding Julia set.  The
advantage of our method is that the techniques based on topologies
and externologies can be applied to any continuous map without
having differentiability or analyticity conditions on the
semi-flow map.

The  discretization processes (the first return Poincar\'e map,
the discretization of a continuous semi-flow, \ldots) and
anti-discretization constructions (suspension, prolongation and
telescopes, see \cite{FHR14}) can also be analyzed by means of
externologies associated to continuous and discrete semi-flows.
This will permit us to connect the theory of basins of end points
developed in \cite{GHR12} for continuous flows with the basins of
end points of discrete semi-flows.

Finally, it is important to say that the problem of sensibility
with respect to the initial conditions can be studied by taking
the externology of open subsets that capture the ``future part" of
``tubes" generated by trajectories of  neighborhoods at each
point.


\begin{thebibliography}{10000}

%\bibitem{Bhatia} \textsc{N. P. Bhatia, G.P. Szego}.
%\emph{Stability Theory of Dynamical Systems}, Springer-Verlag,
%Berlin-Heidelberg-New York, 1970.

\bibitem{Birkhoff} \textsc{G. D. Birkhoff},
\emph{Dynamical Systems}, AMS, Colloquium Publications, vol. 9,
1927.



%\bibitem{B1993}\textsc{W. Bergweiler}, \emph{Iteration of meromorphic functions}, Bul l. Amer. Math. Soc. (N. S.) 29 (1993), 151–188.

%\bibitem{B1984}\textsc{P. Blanchard}, \emph{Complex analytic dynamics on the Riemann sphere}, Bul l. Amer. Math. Soc.
%(N. S.) 11 (1984), 85–141.


\bibitem {DHR09} \textsc{A. {Del R{\'{\i}}o}, L. J. {Hern\'{a}ndez} and M. T.
{Rivas Rodr{\'{\i}}guez}}, \emph{S-Types of global towers of spaces an
exterior spaces}, Appl. Cat. Struct., \textbf{17}, no. 3, 287--301
(2009).

\bibitem{E-H}
 \textsc{D. Edwards and H. Hastings}, \emph{ \v Cech and Steenrod homotopy theories with applications to geometric topology},
 Lecture Notes in Math., \textbf{542}, Springer-Verlag, Berlin, Heidelberg, New York, 1976.


%\bibitem{Engelking} \textsc{R. Engelking}.
%\emph{General topology}, Sigma Series in Pure Mathematics, vol 6,
%Heldelmann Verlag Berlin, 1989.

\bibitem{EHR89}  \textsc{J. I. {Extremiana}, L. J. {Hern{\'a}ndez} and M. T.
{Rivas}}, \emph{An isomorphism theorem of the Hurewicz Type in the proper homotopy category},
Fund. Math., \textbf{132}, 195--214 (1989).

\bibitem {EHR05} \textsc{J. I. {Extremiana}, L. J. {Hern\'{a}ndez} and M. T.
{Rivas}}, \emph{Postnikov factorizations at infinity}, Top. and
its Appl., \textbf{153},  370--393 (2005).

%\bibitem {EHR10} \textsc{J. I. {Extremiana}, L. J. {Hern{\'a}ndez} and M. T.
%{Rivas}}, \emph{An approach to dynamical systems using exterior spaces}, in
%Contribuciones cient\'{\i}ficas en honor de Mirian Andr\'es G\'omez,
% Servicio de Publicaciones, Universidad de La Rioja, Logro\~no, Spain, 2010.


\bibitem {FHR14}  \textsc{J. M. Fern\'andez-Cestau, L. J. Hern\'andez-Paricio, M. T. Rivas-Rodr\'{\i}guez},
\emph{Prolongations, Suspensions and Telescopes}, preprint, 2014.


\bibitem {Freudenthal} \textsc{H. Freudenthal},
\emph{\"Uber die Enden topologisher R\"aume und Gruppen},
Math. Zeith., \textbf{53},  692--713 (1931).

\bibitem{G98} \textsc{J. M. {Garc{\'{\i}}a Calcines}}, \emph{Homotop{\'{\i}}a propia simplicial}, Tesis, La Laguna, 1998.


\bibitem {GGH98} \textsc{J. M. {Garc{\'{\i}}a Calcines}, M.~{Garc{\'{\i}}a Pinillos}
and L. J. {Hern\'{a}ndez}}, \emph{A closed simplicial model category
for proper homotopy and shape theories}, Bull. Aus. Math. Soc.,
\textbf{57}, no.~2, 221--242 (1998).

%\bibitem {GGH01} \textsc{J. M. {Garc{\'{\i}}a Calcines}
%and L. J. {Hern\'{a}ndez}}. \emph{Sequential homology}, Top. and its
%Appl. \textbf{114} /2 , 201-225, (2001).

\bibitem {GGH04} \textsc{J. M. {Garc{\'{\i}}a Calcines}, M. {Garc{\'{\i}}a Pinillos} and  L. J. {Hern\'{a}ndez}},
\emph{Closed simplicial model structures for exterior
and proper homotopy}, Appl. Cat. Struct., \textbf{12}, no.~3,
225--243 (2004).

\bibitem {GHR12} \textsc{J. M. {Garc{\'{\i}}a Calcines}, L. J. {Hern{\'a}ndez} and M. T.
{Rivas Rodr{\'{\i}}guez}}, \emph{Limit and end functors of dynamical systems via
exterior spaces}, Bull. Belg. Math. Soc. Simon Stevin, \textbf{20}, 937--959 (2013).

\bibitem {GHR14} \textsc{J. M. {Garc{\'{\i}}a Calcines}, L. J. {Hern{\'a}ndez} and M. T.
{Rivas Rodr{\'{\i}}guez}}, \emph{A completion construction for  continuous dynamical  systems}, Topological Methods in Nonlinear Analysis (2014).

\bibitem {GHR09} \textsc{M. {Garc{\'{\i}}a Pinillos}, L. J. {Hern\'{a}ndez Paricio} and M. T.
{Rivas Rodr{\'{\i}}guez}}, \emph{Exact sequences and closed model
categories}, Appl. Cat. Struct., \textbf{18}, no.~4, 343--375
(2010).

%\bibitem {H95}  \textsc{L. J. {Hern\'{a}ndez}}.
%\emph{Application of simplicial M-sets to proper homotopy and
%strong shape theories}, Transactions of the AMS,\textbf{ 347}, no.
%2, 363-409, (1995).

%\bibitem{Kererjarto} \textsc{B. Ker\'ekj\'art\'o},\emph{ Vorlesungen uber Topologie},
%vol. 1, Springer-Verlag, 1923.


%\bibitem{Liapunov} \textsc{A. M. Lyapunov}.
%\emph{Probl{\`{e}}me g{\`{e}}n{\`{e}}ral de la stabilit\'e du mouvement},
%Annales de la Faculte des Sciences de Toulouse, 1907.


%\bibitem{Milnor} \textsc{J. Milnor}.
%\emph{Dynamics in One Complex Variable}, Annals of Mathematics Studies-160. Princeton University Press, Princeton, NJ, 2006.

%\bibitem{M}  \textsc{M. Mara{\~n}{\'o}n}, \emph{Exterior discrete semi-flows},
%preprint, 2014.

\bibitem{MR}  \textsc{M. A. Mor\'on, F. R. Ruiz del Portal},
\emph{A note about the shape of attractors of discrete semi-dynamical systems},
Proc. Amer. Math. Soc., \textbf{134}, 2165--2167 (2006).


\bibitem{MSS}  \textsc{M. A. Mor\'on, J. J. S\'anchez Gabites, J. M. R. Sanjurjo},
\emph{Topology and dynamics of unstable attractors},
Fund. Math., \textbf{197}, 239--252 (2007).

\bibitem{Poincare2} \textsc{H. Poincar\'e},
\emph{Les m\'ethodes nouvelles de la m\'ecanique c\'eleste}, Paris, Gauthier-Villars et fils, 1892-99 (1892).

\bibitem{Poincare} \textsc{H. Poincar\'e},
\emph{M\'emoire sur les courbes d\'efinies par une \'equation diff\'erentielle},
Handbook of Algebraic Topology, Chapter 3, 127--167 (1995).

\bibitem{Salamon}\textsc{J. W. Robbin and D. Salamon},
\emph{Dynamical systems, shape theory and the Conley index}, Ergod. Th. and Dynam.
Sys., \textbf{8} (Charles Conley Memorial Volume), 375--393 (1988).

%\bibitem{Sanjurjo}\textsc{J. M. Sanjurjo}.
%\emph{Morse equations and unstable manifolds of isolated invariant
%sets}, Nonlinearity  \textbf{16} 1435-1448, ( 2003).

\bibitem{Gabites} \textsc{J. J. S\'{a}nchez Gabites},
\emph{Dynamical systems and shapes}, Revista de la Real Academia
de Ciencias Exactas, F\'{\i}sicas y Naturales. Serie A: Matem\'{a}ticas
(RACSAM), \textbf{102}, no. 1,  127--160 (2008).

%\bibitem{P95} \textsc{T.Porter}. \emph{Proper homotopy theory},
%Handbook of Algebraic Topology, Chapter 3, pp. 127-167, 1995.

\bibitem{Sanjurjo11} \textsc{J. M. {Sanjurjo}}, \emph{Stability, attraction and shape;
a topological study of flows},  Lecture Notes in Nonlinear Analysis, vol. 12, 93--122 (2011).

%\bibitem {DHR09} \textsc{A. {Del R{\'{\i}}o}, L. J. {Hern{\'a}ndez} and M. T.
%{Rivas Rodr{\'{\i}}guez}}. \emph{S-Types of global towers of spaces an exterior spaces},
%Appl. Cat. Struct. \textbf{17} no. 3, 287-301, (2009).

\end{thebibliography}
\end{document}